\documentclass{article}
\usepackage{amsmath,amssymb,amscd,verbatim, amsthm,graphicx, color}

\newcommand \fk[1]{{{\mathfrak #1}}}
\newcommand \C[1]{{\mathcal #1}}

\newcommand \wti[1]{{\widetilde {#1}}}

\newcommand\fg{\mathfrak g}

\newcommand \bA{{\mathbb A}}
\newcommand \bC{{\mathbb C}}

\newcommand \bH{{\mathbb H}}

\newcommand \bZ{{\mathbb Z}}
\newcommand \bQ{{\mathbb Q}}

\newcommand\ind{{\text {ind}}}

\newcommand\CO{{\C O}}

\newcommand\CU{{\C U}}

\newcommand\ie{{\it i.e.~ }}

\newcommand\etc{{\it etc.~ }}

\newcommand\ep{{\epsilon}}

\newcommand\al{{\alpha}}

\newcommand\fh{{\mathfrak h}}

\newcommand\fz{{\mathfrak z}}
\newcommand\fb{{\mathfrak b}}
\newcommand\fu{{\mathfrak u}}
\newcommand\fl{{\mathfrak l}}
\newcommand\fp{{\mathfrak p}}

\newcommand\ups{{\upsilon}}

\newtheorem{theorem}{Theorem}
\newtheorem*{theorem 2}{Theorem 2}
\newtheorem{corollary}{Corollary}
\newtheorem{lemma}{Lemma}
\newtheorem{proposition}{Proposition}
\newtheorem{definition}{Definition}

\newcommand\rk{\operatorname{rk}}

\numberwithin{equation}{subsection}

\begin{document}

\title
{Multiplicity matrices for the affine graded
  Hecke algebra}

\author{Dan Ciubotaru\thanks{Supported by the NSF grant FRG-0554278.}\\
Department of Mathematics\\ University of Utah\\
Salt Lake City, UT           84112\\
        \texttt{ciubo@math.utah.edu}}

\date{\today}

\maketitle

\begin{abstract}
In this paper we are looking at the problem of determining the
composition factors for the affine graded Hecke algebra via the
computation of Kazhdan-Lusztig type polynomials. We review  the
algorithms of \cite{L1,L2}, and use them in particular to
compute, at every real central character which admits tempered
modules, the geometric parameterization, the Kazhdan-Lusztig
polynomials, the composition series, and the Iwahori-Matsumoto
involution for the representations with Iwahori fixed vectors of the split
$p$-adic groups of type $G_2$ and $F_4$ (and by the nature of the
algorithms, for their Levi subgroups). 
\end{abstract}

\setcounter{tocdepth}{1}
\tableofcontents

\section{Introduction}\label{sec:0}

\subsection{}\label{sec:0.1}
Let $G$ be a complex, connected, simply-connected, semisimple Lie
group, let $H$ be a Cartan subgroup, and $B\supset H$ a Borel
subgroup. Let $\fg,$ $\fb,$ $\fh$ denote the respective Lie algebras,
and let $\Delta,$ $\Delta^+,$ $\Pi$ be the corresponding roots, positive roots, and simple roots
respectively, and let $W$ be the Weyl group. If $\al\in\Delta$, the
corresponding coroot is $\check\al\in\fh,$ and the reflection in $W$ is
$s_{\al}.$ The pairing between $\fh$ and $\fh^*$ is denoted
$\langle\ ,\ \rangle.$ 
 
The {\it affine graded Hecke
algebra} $\bH$ was introduced in \cite{Lu3}. We will only 
consider a special case of the definition, the ``equal parameters'' case. The generators of $\bH$
are the elements $\{t_{s_{\al}}\colon \al\in\Pi\}$  and
$\{\omega\colon \omega\in \fh^*\}.$ Here $\Pi$ denotes the set of
simple roots. As a $\bC$-vector space,
\begin{equation}\label{eq:3.1.1}
\bH=\bC[W]\otimes \bA,
\end{equation}
where
\begin{equation}\label{eq:3.1.2}
\bA=Sym(\fh^*).
\end{equation}
The following commutation relations hold:
\begin{equation}
\omega t_{s_\al}=t_{s_\al} s_\al(\omega)+2\langle\omega,\check\al\rangle,
\quad \al\in\Pi,\ \omega\in\fh^*.
\end{equation}

The center $Z(\bH)$ of $\bH$ consists  of the $W$-invariants in
$\bA$ (\cite{Lu3}):
\begin{equation}
Z(\bH)=\bA^W.
\end{equation}

 On any irreducible $\bH$-module, which is necessarily
finite dimensional, $Z(\bH)$ acts by a {\it central
  character}. Therefore, the central characters are parameterized by
$W$-conjugacy classes in $\fh$. Let $mod_\chi(\Bbb H)$ be the category of finite dimensional modules
of $\Bbb H$ with central character $\chi.$ This will be the main
object of interest in this paper.

\subsection{}\label{sec:BC} The connection with the representation theory of
$p$-adic groups is well-known, and we briefly recall it next. Let $\C G$ be
the split adjoint $p$-adic group whose dual group (in the sense of
Langlands) is $G.$ Let $\C I$ denote an Iwahori subgroup of $\C G$.The
{\it Iwahori-Hecke algebra}, denoted $\C H$
is the algebra of locally constant compactly supported $\C
I$-biinvariant 
functions under convolution.  If a $\C G$-representation $(\pi,V)$ has $\C
I$-fixed vectors, then $\C H$ acts 
on $V^\C I.$ Let $C(\C I,1)$ be  the category of admissible
representations whose every subquotient is generated by its  $\C
I$-fixed 
vectors. 
\begin{theorem}[Borel]\label{t:bc}
The association $V\mapsto V^\C I$ is an equivalence of categories
between $C(\C I,1)$ and the category $mod(\C H)$ of finite dimensional
representations of $\C H.$    
\end{theorem} 
Casselman proved that every subquotient of a minimal principal
series $I(\lambda)$ of $\C G$, where $\lambda$ is an unramified character of
the maximal (split) torus of $\C G$ dual to $H$, is in $C(\C I,1)$, and in fact
every irreducible object in $C(\C I,1)$ is of this form. 

In \cite{KL}, the Langlands (geometric) classification for the category $mod(\C H)$ is proved. 
By a result of Bernstein, the central characters of $\C H$ are 
parameterized by $W$-orbits of elements in $H.$ If $\bar\chi\in H,$
let $mod_{\bar\chi}(\C H)$ be the subcategory of modules with central
character $\bar\chi.$ The connection with
the graded Hecke algebra is in \cite{Lu3}. The algebra $\bH$ is the
associated graded object to a certain filtration in the Iwahori-Hecke
algebra $\C H.$ 
 Assume $\chi$ is a
hyperbolic element of $\fh$, and then $\bar\chi=exp(\chi)$ is a
hyperbolic element of $H.$  There
is a natural correspondence between irreducible objects
\begin{equation}\label{eq:irr}
Irr~mod_{\bar\chi}(\C H)\leftrightarrow Irr~mod_{\chi}(\bH).
\end{equation}
We emphasize that, in (\ref{eq:irr}), $\bar\chi$ is hyperbolic. If the
category is $mod_s(\C H),$ with $s$ arbitrary, one decomposes
$s=s_e\cdot \bar\chi$ into the elliptic and hyperbolic parts. A similar
correspondence holds, 
but the graded Hecke algebra in the right hand side is one for a root
system defined by the centralizer of $s_e$ in $G$ (which is connected
by Steinberg's theorem).

\subsection{} The geometric classification for $mod_\chi(\bH)$ exists as well (\cite{Lu4}; also \cite{KL,Lu3}), and we have standard modules $X$
and irreducible quotients $L.$ 
In the bijections of section \ref{sec:BC}, the standard,  respectively
irreducible, modules correspond.

The classification is expressed in terms of the
geometry of the spaces $\fg_n(\chi):$ 
\begin{equation}G(\chi)=\{g\in G: Ad(g)\chi=\chi\},\quad
\fg_n(\chi)=\{y\in\fg:[\chi,y]=ny\}.\end{equation}

Let $Orb_n(\chi)$ denote the set of $G(\chi)$ orbits on $\fg_n(\chi).$
Assume that $n\in \Bbb Z\setminus\{0\}.$ Then $\fg_n(\chi)$ is a
prehomogeneous $G(\chi)$-vector space (\cite{Ka}), and in fact $Orb_n(\chi)$ is finite. For every $\CO\in Orb_n(\chi),$ $\overline\CO\setminus \CO$ is
  the union of some orbits $\CO'$ with $\dim\CO'<\dim\CO.$

\begin{theorem}[\cite{Lu4}]\label{t:0.1} The standard and irreducible objects in
  $mod_\chi(\Bbb H)$ are in bijection with pairs $\xi=(\CO,\C L)$, where 
\begin{enumerate}
\item $\CO$ is a $G(\chi)$-orbit on $\fg_2(\chi);$
\item $\C L$ is a  $G(\chi)$-equivariant local system on
  $\CO$ of Springer type.
\end{enumerate}
\end{theorem}
More precisely, choose some $e\in \CO.$ Then $\C L$
  corresponds to a representation $\phi$ of the component group
  $A(e,\chi)=G(e,\chi)/G(e,\chi)^0$. The representations $\phi$ which are
  allowed  must be in the
  restriction from $A(e)=G(e)/G(e)^0$ to $A(e,\chi)$ of
  a representation which appears in the Springer correspondence.

\subsection{}\label{sec:0.2}
In this setting, the Kazhdan-Lusztig conjectures take the following
form.

\begin{theorem}[\cite{Lu4}]\label{t:0.2} In (the Grothendieck group of)
  $mod_\chi(\Bbb H)$:

\begin{align}&X_{\xi'}=\sum_{\xi} P_{\xi,\xi'}(1)\cdot L_\xi,\text{
    with }\\\notag
&P_{\xi,\xi'}(q)=\sum_{i\ge 0}\ [\C L: \C H^{2i} IC(\overline {\CO'},\C
    L')\mid_\CO]\cdot q^i,
\end{align}
where $\C H^{j}IC()$ denotes the  $j$-th cohomology sheaf of the intersection
cohomology complex. 
\end{theorem}

(\noindent In the setting of the affine Iwahori-Hecke algebra, the similar result was established in \cite{CG}.)

\begin{corollary}\label{c:0.2}
In particular, $P_{\xi,\xi}=1,$ and if $\xi\neq \xi'$, then
$P_{\xi,\xi'}=0,$ unless $\CO'\subsetneq \overline \CO.$
\end{corollary}

\subsection{}\label{IM}An important feature of $\bH$ is the {\it Iwahori-Matsumoto involution} $IM$,
\begin{equation}
IM(t_w)=(-1)^{\ell(w)}t_{w},\ w\in W,\quad IM(\omega)=-\omega,\ \omega\in \fh^*.
\end{equation}
which gives a bijection
\begin{equation}
IM:\ Irr~mod_\chi\bH\leftrightarrow Irr~mod_{-\chi}\bH.
\end{equation}
One also has an obvious involution
\begin{equation}
\kappa:\bH\to\bH,\quad \kappa(w)=w,\ \kappa(\omega)=-\omega.
\end{equation}
Finally, there is the geometric Fourier-Deligne transform $FD$ (\cite{L2}, 2.1; \cite{EM}) which induces a bijection between irreducible $G(\chi)$-equivariant local systems supported on orbits in $Orb_2(\chi)$ and irreducible $G(\chi)$-equivariant local systems supported on orbits in $Orb_{-2}(\chi).$ The connection between these maps is given by \cite{EM}.

\begin{theorem}\label{t:IM} In $mod_\chi(\bH)$, the Fourier-Deligne transform and the (modified) involution $IM$ induce the same bijection on irreducible modules:
\begin{equation}
\kappa\circ IM=FD.
\end{equation}
\end{theorem}

\subsection{}\label{sec:0.3}
The goal is to compute the
matrix of multiplicities of theorem \ref{t:0.2} for $mod_\chi(\bH).$
We will restrict ourselves (as we may by the theorems exposited in
section \ref{sec:BC}) to the case when $\chi\in\fh$ is hyperbolic
(``real''). We follow the algorithms presented in \cite{L1,L2}. In
sections \ref{sec:1} and \ref{sec:2} we review in a combinatorial way
these algorithms. We should mention that there is no resemblance
between this ``p-adic'' algorithm, and the classical algorithms for computing
Kazhdan-Lusztig polynomials for complex or real groups. One of the
very particular features of this algorithm, at the same time its main
difficulty, is the use of the Fourier-Deligne transform. As
a byproduct of the calculations (and by theorem \ref{t:IM}), one
obtains a (difficult) procedure for computing the Iwahori-Matsumoto involution on the
$\C G$-modules in $Irr~C(\C I,1)$. An explicit, ``closed formula'', description of the
action of $IM$ on $Irr~C(\C I,1)$ is only known for $GL(n),$ by
\cite{Ze3,MW}.

A second feature of the algorithm is that in order
to carry out the calculation for $G$, one needs to have done this first for all
Levi subgroups of $G$.

\smallskip

 Let $\ups$ denote an
indeterminate, which in the end will be specialized to $\ups=1.$ 
We will only consider {\it real}
central characters, \ie $W$-conjugacy classes of hyperbolic
semisimple elements in $\fh.$ The output of the algorithm is a square
matrix of size $\# Irr~mod_\chi(\bH)$ with (polynomial) entries in
$\bZ[v].$ If $c_{\xi,\xi'}(v)$ is such an entry, then the relation
with the
(Kazhdan-Lusztig) polynomial $P_{\xi,\xi'}(q)$ from theorem \ref{t:0.2} is 
\begin{equation}
\varepsilon\cdot c_{\xi,\xi'}(v)=v^{\dim\CO'-\dim\CO}\cdot P_{\xi,\xi'}\left(\frac1{v^2}\right),
\end{equation}
where $\varepsilon\in\{+1,-1\}$ is a sign depending only on $\xi'.$

In other words, the polynomials $c_{\xi,\xi'}$ computed by the
algorithm are, up to sign, those which give conjecturally the degrees
in the Jantzen filtration.

\smallskip

In section \ref{sec:3}, we give some simple examples of how the
algorithm is applied. The regular case, when $\chi=2\check\rho,$ so
that the trivial module is in $mod_\chi(\bH)$, is geometrically
trivial. We present it just as an illustration of the combinatorics of
sections \ref{sec:1} and \ref{sec:2}. For the same reason, we also
present a well-known example from $GL(4)$ (\cite{Ze},\cite{Ze3}). In $GL(n)$,
\cite{Ze1}  relates the polynomials $P_{\xi,\xi'}(q)$  with Kazhdan-Lusztig polynomials
in category $O$, giving therefore a different, indirect, way to
determine them. Finally, we present the cases in $\fg=sp(6)$ and
$\fg=G_2$ where there are cuspidal (in the sense of Lusztig) local
systems. 

In section \ref{sec:4}, we calculate the polynomials when $\fg=F_4,$
for all $\chi$ which are middle elements of nilpotent orbits. By the
geometric classification, they are precisely the (hyperbolic) central
characters which afford tempered modules. These are the difficult
cases of the algorithm. The most interesting case is when $\chi$ is the middle element
of the unique nilpotent orbit in $F_4$ which has a cuspidal local
system. This is presented in more detail in section \ref{sec:3111}.

\smallskip

For parts of these calculations (most notably, Weyl group conjugations
in $F_4$, 
and checking if certain vectors are in the radical of the bilinear form
defined in \ref{sec:1.5}), I used a computer and a computer algebra
system. In all the examples, the notation for nilpotent orbits is as
in \cite{Ca}.

\

\noindent{\bf Acknowledgments.} I thank G. Lusztig for illuminating discussions
about his preprint \cite{L2}, and also P. Trapa for his generous help with Kazhdan-Lusztig
theory in the real and complex groups setting.

\section{Ingredients}\label{sec:1}

We recall the algorithm in \cite{L1,L2} in a combinatorial language. The intention is to express all
the elements of the algorithm purely in terms of the Weyl group and
the roots in $\Delta.$ To this end, we will record two
equivalent descriptions for the same object: the first labeled $(*)$
as in \cite{L2}, and the second, equivalent, but more combinatorial,
labeled $(**)$ (which often appears in \cite{L1} as well).

Throughout sections \ref{sec:1} and \ref{sec:2}, we fix a semisimple element $\chi\in \fh$ such that $\chi$ is the
middle element of a Lie triple in $\fg.$ We assume that $\chi$ is
dominant with respect to $\Delta^+:$ 
\begin{equation}
\langle\al,\chi\rangle\ge 0, \text{ for all }\al\in\Delta^+,
\end{equation}
so that, in particular,  $w_0\chi=-\chi,$ where $w_0$ is the longest Weyl group element.

\subsection{}\label{sec:1.1}
The assumption on $\chi$ is that there exists a Lie algebra
homomorphism 
\begin{equation}
\phi:sl(2,\bC)\to \fg,\quad \phi(\left(\begin{matrix}1&0\\0&-1\end{matrix}\right))=\chi.
\end{equation}
By the representation theory of $sl(2,\bC),$ this implies that the element $\chi$ induces a grading on the Lie algebra $\fg$, by its
$ad$-action: 
\begin{equation}\label{eq:1.1.1}
\fg=\bigoplus_{n\in\bZ}\fg_n,\quad \fg_n=\{x\in\fg: [\chi,x]=nx\}. 
\end{equation}
For any subalgebra $\fp$ of $\fg,$ we will denote similarly
$\fp_n=\fg_n\cap\fp.$ 
We also define
\begin{align}\label{eq:1.1.2}
&r_n(\C R)=\{\al\in\C R: \langle\al,\chi\rangle=n\}, \text{ for any
  subset of roots }\C R\subset \Delta, and\\\notag
&r_n(w)=r_n(\Delta)\cap (w^{-1}\cdot\Delta^+), \text{ for any element
  }w\in W.
\end{align}
Some immediate properties that we will need later are:

\begin{lemma}\label{l:1.1} For all $w\in W,$ $n\in\bZ$:
\begin{enumerate}
\item[(a)]$r_n(w)=r_{-n}(w_0w);$
\item[(b)] $r_n(w)\cap
r_n(w_0w)=\emptyset;$ 
\item[(c)] $r_n(w)\cup r_n(w_0w)=r_n(\Delta).$
\end{enumerate}
\end{lemma}
\begin{proof} Straightforward.\end{proof}

\subsection{}\label{sec:1.2} Define 
\begin{align}\label{eq:1.2.1}
&G(\chi)=\{g\in G: Ad(g)\chi=\chi\},\\\notag
& W(\chi)=\{w\in W: w\chi=\chi\},\\\notag
&\C B(\chi)=G(\chi)\text{-orbits in }\{\fb': \fb'\text{ Borel subalgebra with } \chi\in\fb'\}.
\end{align}
The space in which the constructions will take place is $\C
K(\chi)$, defined as
\begin{align}
(*): &\ \C K(\chi)=\bQ(\ups)\text{-vector space with basis } \C B(\chi), \\\label{eq:1.2.2}\notag
(**): &\ \C K(\chi)=\bQ(\ups)\text{-vector space with basis } W/W(\chi).
\end{align}

The space $\C K(\chi)$ has two  involutions that we consider. The first is 
\begin{equation}\label{eq:1.2.3}
\beta :\C K(\chi)\to \C K(\chi),\text{ determined by
}\beta(\ups)=\ups^{-1}, \text{ and the identity on the $\bQ(v)$-basis.}
\end{equation}

The second involution $\sigma$ associates to
each Borel subalgebra, the opposite Borel subalgebra, and it is the
identity on $\bQ(v).$  On $W(\chi)$-cosets, this
is determined by
\begin{align}
(**):\  &\sigma(w)=w_0\cdot w \text{ and } \sigma(v)=v.
\end{align}

\subsection{}\label{sec:1.3} 
Consider the variety
\begin{equation}
\C S_\chi=\{(\fb',\fb''):\chi\in\fb'\cap\fb''\},
\end{equation}
with the natural diagonal action of $G(\chi).$
One defines the space $(\C B\times \C B)(\chi)$:
\begin{align}\label{eq:1.3.1}
(*): &\ (\C B\times \C B)(\chi)=G(\chi)\text{-orbits on } \C S_\chi,
  \text{ which can naturally be identified with}\\\notag
(**): &\ (\C B\times \C B)(\chi)=(W\times W)/W(\chi)\ (\text{where $W(\chi)$ is regarded
  as the diagonal subgroup}).
\end{align}
This space is equipped with a function 
\begin{equation}\label{eq:1.3.1a}
\tau:(\C B\times \C B)(\chi)\to\bZ
\end{equation} 
as follows. For a Borel 
subalgebra $\fb',$ let $\fu'$ denote the unipotent
radical. Then
\begin{align}\notag
(*):&\ \tau((\fb',\fb''))=-\dim\displaystyle\frac{\fu_0'+\fu_0''}{\fu_0'\cap\fu_0''}+\dim\displaystyle\frac{\fu_2'+\fu_2''}{\fu_2'\cap\fu_2''},
  \text{ or, equivalently,}\\\label{eq:1.3.2}
(**):&\ \tau((w_1,w_2))=\#(r_2(w_1)\vee
r_2(w_2))-\#(r_0(w_1)\vee r_0(w_2)),
\end{align}
where $\vee$ denotes the symmetric difference set
operator. It is proved in \cite{L2} that $\tau$ is the same if one
replaces $2$ in the formulas above by $-2.$

\subsection{}\label{sec:1.4}
Set
\begin{equation}
c=\# r_2(\Delta)-\# r_0(\Delta),
\end{equation}
where $r_n$ is defined in (\ref{eq:1.1.2}).

\begin{lemma}\label{l:1.3}
\begin{enumerate}
\item[(a)] The map $\tau$ in definition $(**)$ of (\ref{eq:1.3.2}) is
  well-defined, \ie 
\begin{equation*}\tau((w_1w,w_2w))=\tau((w_1,w_2)), \text{ for any } w\in W(\chi).\end{equation*} 
\item[(b)] Since $w_0\chi=-\chi,$
  $\tau((w_1,w_2))=\tau((w_1w_0,w_2w_0)),$ for every $(w_1,w_2)\in
W\times W.$
\item[(c)] For any $w_1,w_2\in W$,
$$\tau((w_1,w_2))+\tau((\sigma(w_1),w_2))=c.$$ (So
this sum is independent of $w_1,w_2$).
\end{enumerate}
\end{lemma}

\begin{proof}
Part (a) is immediate. For part (b), one uses lemma \ref{l:1.1}.(a)
and the observation after the equation (\ref{eq:1.3.2}). Part (c)
follows from lemma \ref{l:1.1}, parts (b) and (c).
\end{proof}

\subsection{}\label{sec:1.5}Let \begin{equation}pr_j:(\C B\times \C
  B)(\chi)\to \C B(\chi),\ \ j=1,2\end{equation} denote the
projection onto the $j$-th coordinate. One defines a symmetric
bilinear form on $\C K(\chi)$
$$(\ :\ ):\C K(\chi)\times \C K(\chi)\to \bQ(\ups)$$ by
\begin{align}
(*): &\ \ e_\chi^{-1}\cdot([\fb']:[\fb''])&=\displaystyle\sum_{\begin{matrix}\Omega\in (\C
    B\times \C B)(\chi)\\pr_1\Omega=[\fb'], pr_2\Omega=[\fb'']
  \end{matrix}} (-\ups)^{\tau(\Omega)},\\
(**): &\ \ e_\chi^{-1}\cdot([w_1]:[w_2])&=\displaystyle\sum_{[(w',w'')]\in
  pr^{-1}_1([w_1])\cap pr^{-1}_2([w_2])} (-\ups)^{\tau((w',w''))}\\\notag
                     &&=\displaystyle\sum_{w\in
                       [w_1]}(-\ups)^{\tau((w,w_2)}=\displaystyle\sum_{w\in[w_2]}(-\ups)^{\tau((w_1,w))}.  
\end{align}
In these formulas, $[\bullet]$ denotes the class of an element
$\bullet$ in $\C B(\chi)$, and similarly in $(\C B\times \C B)(\chi).$

The factor $e_\chi\in\bQ(v)$ is a normalization factor, and it
depends only on $\chi.$
The choice that we will use is $e_\chi=(1-v^2)^{-\rk\fg}$ (\cite{L1}), so that we have the identity
\begin{equation}\label{eq:1.5.3}
\beta((\beta(\xi),\beta(\xi')))=(-1)^{\rk\fg}(-v)^{2\rk\fg-c}(\sigma(\xi):\xi'),\text{
  for all }\xi,\xi'\in\C K(\chi).
\end{equation}

In general, the bilinear form $(\ :\ )$ is degenerate. Let
$Rad$ denote its radical. From equation (\ref{eq:1.5.3}), we see that
\begin{equation}\beta(Rad)=Rad.\end{equation}

\begin{proposition}[\cite{L2}] The dimension of $\C K(\chi)/Rad$
  equals $\# Irr~mod_\chi(\bH),$ the number of inequivalent, irreducible representations with
  central character $\chi.$ 
\end{proposition}

\subsection{Examples}\label{sec:1.6} To illustrate the definitions so
far, we give  some examples.  In the following tables, for
simplicity, we will take the normalization factor $e_\chi=1$. 

\subsubsection{$\mathbf{A_2}$} The simple roots are
$\{\ep_1-\ep_2,\ep_2-\ep_3\},$ and let the simple reflections be
denoted by $s_1$ and $s_2$. Then
$W=\{1,s_1,s_2,s_1s_2,s_2s_1,s_1s_2s_1\}.$ We look at
$\chi=2\check\rho=(2,0,-2).$ The bilinear form is

\begin{center}
\begin{tabular}{l||l|l|l|l|l|l|}
&$1$ &$s_1s_2s_1$ &$s_1s_2$ &$s_2s_1$ &$s_1$ &$s_2$\\
\hline
\hline
$1$ &$1$ &$\ups^2$ &$-\ups$ &$-\ups$ &$-\ups$ &$-\ups$\\
\hline
$s_1s_2s_1$ &$\ups^2$ &$1$ &$-\ups$&$-\ups$&$-\ups$&$-\ups$\\
\hline
$s_1s_2$ &$-\ups$ &$-\ups$ &$1$ &$\ups^2$ &$\ups^2$ &$1$\\
\hline
$s_2s_1$ &$-\ups$  &$-\ups$ &$\ups^2$ &$1$ &$1$ &$\ups^2$\\
\hline
$s_1$ &$-\ups$ &$-\ups$ &$\ups^2$ &$1$ &$1$ &$\ups^2$\\
\hline
$s_2$ &$-\ups$ &$-\ups$ &$1$ &$\ups^2$ &$\ups^2$ &$1$\\
\hline
\end{tabular}\ .
\end{center}

In this case, $\dim(Rad)=2$, and a basis is given by $\{s_1-s_2s_1,\
s_2-s_1s_2\}.$  This is a particular case of section \ref{sec:3.1}.

\subsubsection{$\mathbf {C_2,\ \chi=(1,1)}$.} The simple roots are
$\{\ep_1-\ep_2,2\ep_2\}.$ We look at $\chi=(1,1),$ the middle
nilpotent element of the nilpotent orbit $(22)$ in $\fk{sp}(4,\bC).$
Then $W(\chi)=\{1,s_1\},$ and
$W/W(\chi)=\{[1],[s_2],[s_1s_2],[s_2s_1s_2]\}.$ The bilinear form is 

\begin{center}
\begin{tabular}{l||l|l|l|l|}
&$[1]$ &$[s_2]$ &$[s_1s_2]$ &$[s_2s_1s_2]$\\
\hline
\hline
$[1]$ &$1+\ups^{-2}$ &$-\ups^{-1}-\ups$ &$1+\ups^2$ &$-\ups-\ups^3$\\
\hline
$[s_2]$ &$-\ups^{-1}-\ups$ &$2$ &$-2\ups$ &$1+\ups^2$\\
\hline
$[s_1s_2]$ &$1+\ups^2$ &$-2\ups$ &$2$ &$-\ups^{-1}-\ups$\\
\hline
$[s_2s_1s_2]$ &$-\ups-\ups^3$ &$1+\ups^2$ &$-\ups^{-1}-\ups$
&$1+\ups^{-2}$\\
\hline
\end{tabular}\ .
\end{center}

In this example, the form is nondegenerate.

\subsection{}\label{sec:1.8} We have defined the space $\C K(\chi)$
equipped with a degenerate bilinear form. The final basic ingredient of the algorithm is an {\it
  induction} map. Let $\fp$ be a parabolic subalgebra of $\fg,$ such
that $\chi\in \fp.$ The parabolic $\fp$ is not necessarily standard with
respect to the fixed Borel $\fb$ (\ie the choice of positive roots $\Delta^+$).  Let $\fl$ denote the Levi component of
$\fp,$ and let $\fu_p$ be the nilradical. 

One can define $\C K_\fl(\chi)$,
$W_{\fl}(\chi)$ \etc similarly to the definitions for $\fg$ in the
previous sections. Let $proj:\fp\to\fl$ denote the
projection onto the Levi factor. 

The induction map is defined as
\begin{equation}\label{eq:1.8.1}
(*): \ \ind_{\fp}^{\fg}: \C B_{\fl}(\chi)\to \C B(\chi),\ \ind_{\fp}^{\fg}(\fb')=proj^{-1}(\fb').
\end{equation}
One can define the same map in terms of the Weyl group.
The roots in $\fu_p$ are a subset of
$\Delta,$ but not necessarily of $\Delta^+.$ Let $w_p$ be
a Weyl group of minimal length such that $w_p(\Delta(\fu_p))\subset
\Delta^+.$ Then
\begin{equation}\label{eq:1.8.2}
(**):\ \ind_{\fp}^{\fg}: \C B_{\fl}(\chi)\to \C B(\chi),\ \ind_{\fp}^{\fg}([w])=[w\cdot w_p^{-1}].
\end{equation}

\section{Bases}\label{sec:2} 

The goal is to construct two pairs of bases $(\C Z_+,\C U_+)$ and $(\C
Z_-,\C U_-)$ for $\C K(\chi)/Rad.$  
The definition is inductive. 
The construction of the bases $\C Z_+,\C U_+,$
respectively $\C Z_-,\C U_-$ is done in parallel in the space $\C K(\chi)$, so we will use the
subscript $\pm$ for simplicity when there is no risk of confusion.

In the end,
the change of bases matrix for the pair $(\C Z_+,\C U_+)$
(equivalently for $(\C Z_-,\C U_-)$) is the desired {\it multiplicity
  matrix}.

\subsection{}\label{sec:2.1} The standard modules with central character $\chi$ (and
so the sets $\C Z_\pm$ and $\C U_\pm$) are parameterized, as in theorem
\ref{t:0.1}, by $Orb_{\pm 2}(\chi)$, the $G(\chi)$-orbits on $\fg_{\pm 2}$ and local systems. \cite{L1} gives a
parameterization of the orbits in terms of certain parabolic
subalgebras of $\fg$. We recall next the parameterization of orbits in
$\fg_2.$ (The case of $\fg_{-2}$ is absolutely analogous.)
Let $e$ be a (nilpotent) representative of an orbit
$\CO=\CO_e$ of $G(\chi)$ in $\fg_2.$

By the graded version of the Jacobson-Morozov triple (\cite{L1}),
$e\in \fg_2$ can be
embedded into a Lie
triple $\{e,h,f\}$, such that $h\in \fh\subset\fg_0,$ and $f\in
\fg_{-2}.$ From the pair of semisimple elements $\chi$ and $h$, one
can  define two associated parabolic subalgebras $\fp_\pm$ as in
\cite{L1}.

Define a gradation of $\fg$ with respect to $h$ as well,
\begin{equation}\fg^r=\{y\in\fg: [h,y]=ry\},\ r\in\bZ,
\end{equation}
 and set 
\begin{equation}
\fg^r_t=\fg_t\cap\fg^r.
\end{equation}
 Then
\begin{equation}\fg=\bigoplus_{t,r\in\bZ}\fg^r_t.
\end{equation}
Set
\begin{align}\label{par1}&\fl=\bigoplus_{t=r}\fg_t^r,\quad
  \fu_-=\bigoplus_{t<r}\fg_t^r,\quad \fu_+=\bigoplus_{t>r}\fg_t^r,\\\notag
  &\fp_+=\fl\oplus\fu_+, \quad\fp_-=\fl\oplus\fu_-.\end{align}
Since we want to emphasize the nilpotent element $e$, we will write in
this section
$\fp^e=\fp_+$ and similarly $\fl^e,\fu^e.$  
Clearly, $\fh\subset\fg_0^0\subset\fl^e.$

\begin{definition}
One says that $\chi$ is {\it rigid} for a Levi subalgebra $\fl,$ if
$\chi$ is congruent modulo the center $\fz(\fl)$ to a middle element of a
nilpotent orbit in $\fl$. 
\end{definition}

We record the important properties of $\fp^e.$ The centralizer of an element $t$ in a group $Q$ will be denoted below by $Z_Q(t),$ and its group of components by $A_Q(t).$

\begin{proposition}[\cite{L1}]\label{p:2.1} Consider the subalgebra $\fp^e$ defined by
  (\ref{par1}), and let $P^e$ be the corresponding parabolic subgroup.

\begin{enumerate}

\item $\fp^e$ depends only on $e$ and not on the entire Lie triple
  $\{e,h,f\}.$

\item $\chi$ is {rigid} for $\fl^e.$ 

\item $e$ is an element of the
  open $L^e(\chi)$-orbit in $\fl^e_2.$ 

\item The $P(\chi)^e$-orbit of $e$ in $\fp^e_{2}$ is open, dense in $\fp^e.$ 

\item  $Z_{G(\chi)}(e)\subset P^e$.

\item The inclusion
  $Z_{L^e(\chi)}(e)\subset Z_{G(\chi)}(e)$ induces an isomorphism
  of the component groups \begin{equation}A_{L^e(\chi)}(e)\cong A_{G(\chi)}(e).\end{equation}
\end{enumerate}

\end{proposition}

\noindent{\bf Remark.} Note that, since $\chi$ is rigid in $\fl^e,$ the component group
$A_{L^e(\chi)}(e)$ is the component group corresponding to a nilpotent orbit in
$\fl^e$, and these are all well-known (see \cite{Ca}). In conclusion, part (6) of
  proposition \ref{p:2.1} gives an effective way to compute the
  component groups for the  orbits in $Orb_2(\chi).$

\medskip

In addition, an immediate corollary of (4) and (5) in  proposition \ref{p:2.1} is a dimension
formula for the orbits in $Orb_2(\chi).$

\begin{corollary}[\cite{L2}]\label{c:2.1} For an orbit $\CO_e\in
  Orb_2(\chi),$ \begin{equation}\dim \CO_e=\dim
    \fp_2^e-\dim\fp_0^e+\dim\fg_0,\end{equation}
where $\fp^e_i=\fp^e\cap\fg_i,$ $i=0,2.$
\end{corollary}

\begin{definition} A parabolic subgroup $P$ with Lie algebra $\fp$ is
  called {\it good} for $\chi$ if $\fp=\fp^e$ for some
  nilpotent $e\in \fg_2$ (notation as in (\ref{par1})), and such that
  it satisfies (2) in  proposition \ref{p:2.1}.
\end{definition}

Let $\C P(\chi)$ denote the set of good parabolic
subgroups for $\chi.$ The parameterization of $Orb_2(\chi)$ is as follows.

\begin{theorem}[\cite{L1}]\label{t:2.5}
The map $\CO_e\mapsto P^e$ defined above
induces a bijection between $Orb_2(\chi)$ and $G(\chi)$-conjugacy
classes in $\C P(\chi)$.
\end{theorem}

\subsection{}\label{sec:2.1a} Perhaps, it is better to think that $Orb_{\pm 2}(\chi)$
are parameterized by a set of pairs $(\fp,h),$ where $h$ is a middle
element of a nilpotent orbit in $\fl$, and $\fp$ is a good
parabolic. Let us call $\C E(\chi)$ this parameter set.

In general, for computations, we will apply the following equivalent
(but inelegant) procedure to determine $\C E(\chi)$. Let $\fl$ be a standard Levi subalgebra
(corresponding to a subset of the simple roots $\Pi$), and let $h$ be
a middle element of a Lie triple $\{e,h,f\}$ for $\fl$, assumed dominant for
$\Delta^+(\fl).$ For a given $\fg,$ there are finitely many pairs
$(\fl,h)$ like this. 
Let $\fz(e,h,f)$
denote the centralizer of $\{e,h,f\}$ in $\fg.$ 
If
\begin{equation}\label{eq:2.1.1}
(**): \text{there exists }w\in W \text{ such that }w(h+\nu_0)=\chi
\text{ for some }\nu_0 \in \fz(e,h,f)\cap\fh,
\end{equation}
then we set 
\begin{equation}
s=w\cdot h,
\end{equation}
and we define the two parabolic subalgebras $\fp_{s,-}$ and $\fp_{s,+}$
associated to $s,\chi$ as in (\ref{par1}), with common Levi subalgebra
$\fl_s=\fp_{s,+}\cap\fp_{s,-}.$ The pair
$(\fp_{s,+},s)$ parameterizes an orbit in $Orb_2(\chi)$, and similarly
$(\fp_{s,-},s)$ parameterizes an orbit in $Orb_{-2}(\chi).$ This is
how all the orbits are indexed, in other words, our sets $\C E(\chi)$
are formed of such pairs $(s,\fp_{s,\pm})$. Note, that by corollary \ref{sec:2.1},
one can compute the dimension of the associated orbit at once.

\subsection{}\label{sec:2.3} We retain the notation from the previous subsections.

Consider $\CO\in Orb_{\pm 2}(\chi)$, and let $\fl_s,$ $\fp_{s,+}$,
$\fp_{s,-}$ be the corresponding
subalgebras defined in section \ref{sec:2.1a}. The
bases $\C Z_{\pm}$ and $\C U_{\pm}$ are partitioned as:
\begin{equation}\label{eq:2.2.1}
\C Z_{\pm}=\displaystyle{\sqcup_{(s,\fp_{s,\pm})\in\C E(\chi)}} \C Z_{\pm}(\CO),\quad \C U_{\pm}=\displaystyle\sqcup_{(s,\fp_{s,\pm})\in\C E(\chi)} \C U_{\pm}(\CO).
\end{equation}

 If
$\fl_s=\fg,$ then $\CO=\CO_m$ is necessarily the unique maximal nilpotent 
orbit in $Orb_{\pm 2}(\chi).$

Assume that $\CO\neq\CO_m.$ Then $\fl_s$ is a proper Levi subalgebra
of $\fg.$ By construction, there exists a Lie triple $(e',s,f')$ of
$\CO$, such that $(e',s,f')\subset \fl_s.$ Let $\CO_{\fl_s}$ denote
the nilpotent orbit of $e'$ in $\fl_s.$ By induction, we can assume
that the bases $\C Z^{\fl_s}_{\pm}(\CO_{\fl_s})$ corresponding to
central character $s$ are constructed for $\fl_s.$ 

Then
\begin{equation}\label{eq:2.2.2}
\C Z_+(\CO)=\ind_{\fp_{s,+}}^\fg(\C Z_+^{\fl_s}(\CO_{\fl_s})),\quad
\C Z_-(\CO)=\ind_{\fp_{s,-}}^\fg(\C Z_-^{\fl_s}(\CO_{\fl_s})).
\end{equation}

We recall that the elements in each set $\C Z_\pm(\CO)$ are parameterized by
certain local systems, or equivalently certain representations of the group of components $\widehat
A_{G}(\chi,e).$ 

\begin{proposition}[\cite{L2},2.17]\label{p:2.2} If $\CO\neq\CO',$ and if $(\xi,\xi')\in\C
  Z_+(\CO)\times\C Z_+(\CO')$ or  $(\xi,\xi')\in\C
  Z_-(\CO)\times\C Z_-(\CO')$, then 
\begin{align}\label{eq:2.2.3}
(\xi:\xi')=0.
\end{align}
\end{proposition}

\subsection{} 
Let us denote 
\begin{equation}\label{eq:2.3.1}
\C Z'_{\pm}=\C Z_\pm\setminus\C Z_\pm(\CO_m), \quad \C U'_{\pm}=\C
U_\pm\setminus\C U_\pm(\CO_m). 
\end{equation}

The multiplicity matrix computed by the algorithm is a matrix with
coefficients in $\bZ[\ups]$,
\begin{align}\label{eq:matrix}
\C N=\left(\begin{tabular}{c|c} $\C N_{1,1}$&$\C N_{1,2}$\\\hline$ \C
    N_{2,1}$&$\C N_{2,2}$\end{tabular}\right),
\end{align}
where 

\begin{enumerate}
\item $\C N_{1,1}$ is an upper unitriangular matrix of size $\#\C
  Z_{\pm}'\times \#\C Z_{\pm}'$ which will be computed in
  equation (\ref{eq:2.3.6}), 
\item $\C N_{1,2}$ is a matrix of size $\#\C Z_{\pm}'\times \#\C
Z_{\pm}(\CO_m)$ computed in equation (\ref{eq:2.4.4}), 
\item $\C N_{2,1}$ is the zero matrix of size $ \#\C
Z_{\pm}(\CO_m)\times \#\C Z_{\pm}',$
\item $\C N_{2,2}$ is the identity matrix of size $ \#\C
Z_{\pm}(\CO_m)\times \#\C
Z_{\pm}(\CO_m).$ 
\end{enumerate}

\medskip

The sets $\C Z'_{\pm}$ were constructed by induction in section \ref{sec:2.3}. One sets a partial
ordering $\le$ on $\C Z'_{\pm}$ given by the dimensions of the
corresponding orbits. In this order, the unique element in $\C Z_\pm (0)$
is the minimal element.

Now we explain
the construction of $\CU'_{\pm}$. 
Define the matrices 
\begin{align}\label{eq:2.3.2}
\C M_\pm=((\xi:\xi'))_{\xi,\xi'\in\C Z'_{\pm}}.
\end{align}
By proposition \ref{p:2.2}, these matrices are block-diagonal, with
blocks of sizes $\#\C Z_\pm(\CO).$ 

\begin{lemma}[\cite{L2},1.11,3.7]\label{l:2.3} The matrices $\C
  M_{\pm}$ are invertible.
\end{lemma}

For every $\xi\in Z_\pm'$, we find the vector 
\begin{equation}\label{eq:2.3.3}
V_\xi=(a'_{\xi,\xi'})_{\xi'\in Z_\pm'}=\C M_\pm^{-1}\cdot
((\beta(\xi),\xi')_{\xi'\in Z_\pm'}. 
\end{equation}
By lemma 1.13 in \cite{L2}, $a'_{\xi,\xi}=1,$ and $a'_{\xi,\xi'}=0$
unless $\xi'\le\xi.$ Moreover, from \cite{L2}, 1.14,
\begin{equation}\label{eq:2.3.4}
\beta(V_\xi^T)\cdot V_{\xi'}=\left\{\begin{matrix} 1,&\text{ if }
    \xi=\xi'\\
                                                 0,&\text{ if
                                                 }\xi\neq\xi'\end{matrix}\right., 
\end{equation} 
where $V^T$ denotes the transpose of $V$.

\begin{proposition}[\cite{L2}]\label{p:2.3} There exists a unique
  family $\{c_{\xi,\xi'}:\ \xi,\xi'\in Z_\pm'\}\}$ such that 
\begin{enumerate}
\item[(i)] $c_{\xi,\xi}=1$, $c_{\xi,\xi'}=0$ if $\xi'\not\le\xi$, and
  $c_{\xi,\xi'}\in \ups\bZ[\ups]$ if $\xi'<\xi;$
\item[(ii)] $c_{\xi,\xi'}=\displaystyle\sum_{\xi''\in\C Z'_\pm}\beta(c_{\xi,\xi''})~a'_{\xi'',\xi'}.$
\end{enumerate}
Set 
\begin{align}\label{eq:2.3.5}
\mu_\xi=\displaystyle\sum_{\xi'\in\C Z_\pm'}c_{\xi,\xi'}\xi'.
\end{align}
Then $\C U_\pm'=\{\mu_\xi:\xi\in\C Z_\pm'\}.$
\end{proposition}

In other words, in the multiplicity matrix,
\begin{equation}\label{eq:2.3.6}
\C N_{1,1}=(c_{\xi,\xi'})_{\xi,\xi'\in\C Z_\pm'}.
\end{equation}

\subsection{} It remains to explain the computation of the sets $\C
Z_+(\CO_m)$ and $\C U_+(\CO_m).$ (The other pair is computed in the
obvious analogue way.)

Since $\C K(\chi)$ has a symmetric bilinear form, for every subspace
$\C W\subset \C K(\chi),$ we can define the orthogonal complement $\C
W^\perp.$ Clearly, $Rad\subset \C W^\perp.$

Let $\C W_+$ be the subspace spanned by $\C Z_+'.$ In fact, $\C Z'_+$ is
a basis of $\C W_+.$ Define the projections $Y_+$, respectively
$Y_+^\perp$ of $\C K(\chi)$ onto $\C W_+$, respectively $\C
W_+^\perp.$ Explicitly, 
\begin{equation}
Y_+^\perp(x)=x-Y_+(x),
\end{equation}
 where 
\begin{equation}
Y_+(x)=\displaystyle\sum_{\xi\in\C Z_+'}a_{x,\xi}\xi,\quad\text{and }
(a_{x,\xi})_{\xi\in \C Z_+'}=\C M_+^{-1}\cdot ((x:\xi'))_{\xi'\in \C Z_+'}.
\end{equation}

\begin{proposition}[\cite{L2}]\label{p:2.4}
Let $J_-$ be defined by 
\begin{equation}
J_-=\{\xi_0\in \C
Z_-':\ Y_+^\perp(\mu_{\xi_0})\notin Rad\}.
\end{equation}
The sets
$\C Z_+(\CO_m)$ and $\C U_+(\CO_m)$ are then obtained as follows:
\begin{equation}
\C Z_+(\CO_m)=\{\xi=Y_+^\perp(\mu_{\xi_0}):\ \xi_0\in J_-\},\quad 
\C U_+(\CO_m)=\{\mu_\xi=\mu_{\xi_0}:\ \xi_0\in J_-\}.
\end{equation}

\end{proposition}

This concludes the construction of the bases. To complete the matrix
of multiplicities, one finds
\begin{equation}\label{eq:2.4.4}
\C N_{1,2}=(c_{\xi,\xi'})_{\xi\in\C Z_+(\CO_m),\xi'\in\C Z_+'}=\C
  M_+^{-1}\cdot ((\mu_\xi:\xi''))_{\xi\in\C Z_+(\CO_m),\xi''\in\C Z_+'}.
\end{equation}

\medskip

\noindent{\bf Remarks.}

(1) The transformation $Y_+^\perp$ encodes the Fourier-Deligne
transform $FD$ (see \cite{L2}), and the essential fact in the construction
of the proposition is that the $FD$ dual of a local
system on the open orbit in $\fg_2$ is a local system
which {\it does not} live on the open orbit in $\fg_{-2}.$ By theorem \ref{t:IM}, the
equivalent representation theoretic statement is that the
Iwahori-Matsumoto dual of a tempered module is not tempered.

(2) Always, the basis element corresponding to the zero orbit in $\C
Z_-'$,  $\xi_{triv}\in \C Z_-'$, is in $J_-.$ In fact,
$\xi=Y_+^\perp(\xi_{triv})\in \C Z_+(\CO_m)$ corresponds to the
trivial local system on $\CO_m.$ The equivalent, representation
theoretic statement is a combination of two facts: firstly, that the Iwahori-Matsumoto involution of the
generic module is the spherical module, and secondly, that the generic
module is parameterized by the trivial local system on $\CO_m$ (\cite{BM,Re}). 

(3) To compute $FD$ in general (not just for the elements supported on the open orbit), one can use the following procedure. Assume $\xi\in \C Z_+$ corresponds to $(\CO,\C L)$ and $\xi'\in \C Z_-$ corresponds to $(\CO',\C L').$ Then $FD(\C L)=\C L'$ if and only if $\mu_\xi\in \CU_+$ is the (unique) element of $\CU_+$ such that 
\begin{equation}
\{Y^+(\mu_{\xi'})\}\cup (\CU_+\setminus\{\mu_\xi\})
\end{equation}
is a linear independent set (actually a basis) of $\C K(\chi)/Rad.$

\section{Examples: the regular case, $gl(4)$, $sp(4)$, $sp(6),$ and $G_2$}\label{sec:3}

In the explicit examples in $gl(4),$ $sp(4),$ $sp(6)$, $G_2,$ the symbol used to denote the $G(\chi)$-orbits on $\fg_2$ and the local systems encodes  the dimension of the  orbit. When there are more orbits with the same dimension, we add an subscript $a,b,\dotsc.$ If the component group is not trivial, then in these examples it is always $\bZ/2\bZ$, and we add a subscript $t$ or $s$ corresponding to the trivial, respectively the sign representations.

\subsection{Regular central character}\label{sec:3.1}

 Recall that $\Pi\subset \Delta^+$ denotes the set of simple roots, and
 fix root vectors $X_\al.$ 
When $\chi=2\check\rho,$ the orbits and Kazhdan-Lusztig polynomials have an
especially simple form:
\begin{align}
\fg_2(\chi)&=\bigoplus_{\al\in\Pi}\bC\cdot X_\al,\\\notag
G(\chi)&=\text{ the Cartan subgroup }H.
\end{align} 
There is a one-to-one correspondence 
\begin{equation}
Orb_2(\chi)\leftrightarrow 2^\Pi,
\end{equation}
where to every $\Pi_M\subset\Pi$ we associate the orbit
$\CO_M=\sum_{\al\in\Pi_M}\bC^*\cdot X_\al.$ All the orbits have smooth
closures, and only trivial local systems appear, and therefore, all
Kazhdan-Lusztig polynomials are either $0$ or $1$, depending on the
closure ordering. The closure ordering is given by the inclusion of
subsets of $\Pi.$ 

\smallskip

We include however the combinatorial calculation using the algorithm
explained in section \ref{sec:1} and \ref{sec:2}, just for the purpose to
illustrate the elements of this algorithm. 

\medskip

If
$\chi=2\check\rho$, then $r_2(\chi,\Delta)=\Pi$ and
$r_0(\chi,\Delta)=\emptyset.$

Let $\Pi_M$ be a subset of $\Pi,$ and $w_0(M)$ be the longest Weyl
element in $W(M).$ 

\begin{lemma}\label{l:1}\  

(1) If $\Pi_M\subset \Pi,$ then $r_2(\chi,w_0(M))=\Pi\setminus\Pi_M.$

(2) For any $\Pi_{M_1},\Pi_{M_2}\subset \Pi,$
    $$\tau(w_0(M_1):w_0(M_2))=|\Pi_{M_1}\vee \Pi_{M_2}|.$$
In particular, $\tau(w_0(M),1)=|\Pi_M|$ and
$\tau(w_0(M),w_0)=|\Pi|-|\Pi_M|.$ 
\end{lemma} 

\begin{proof}
It follows immediately from the fact that
$r_2(\chi,w_0(M))=\{\al\in\Pi:
w_0(M)\al\in\Delta^+\}=\Pi\setminus\Pi_M.$ 
\end{proof}

Since $W(\chi)=\{1\},$ the bilinear form in this case is
$(w_1:w_2)=(-v)^{\tau(w_1,w_2)},$ for every $w_1,w_2\in W.$ The
radical can also be easily described. For every $\Pi_M\subset\Pi,$
define
\begin{equation}
S_M=\{w\in W: w\al\notin\Delta^+,\forall\al\in\Pi_M\text{ and
}w\beta\in\Delta^+,\forall \beta\in\Pi\setminus\Pi_M\}.
\end{equation}
Note that $w_0(M)\in S_M.$ Then it follows immediately that a basis
for $Rad$ is
\begin{equation}
\bigsqcup_{\Pi_M\subset\Pi}\{w_0(M)-w': w'\in S_P, w'\neq w_0(M)\}.
\end{equation}
(Note that $\dim \C K(\chi)=|W|,$ while $\dim\C K(\chi)/Rad=2^{|\Pi|}.$)

\begin{definition}
For every subset $\Pi_{M'}\subset\Pi,$ define
$\xi_{M'}=\displaystyle{\sum_{\Pi_M\subset\Pi_{M'}}}
  v^{|\Pi_{M'}|-|\Pi_{M}|} w_0(M).$ 
\end{definition}

\begin{proposition}\label{p:1} If  $\Pi_{M'}\subset\Pi,$ we have the
  following identities:

(1) $(\xi_{M'},\xi_{M'})=\displaystyle{\sum_{\Pi_M\subset\Pi_{M'}}(-v^2)^{|\Pi_{M'}|-|\Pi_M|}=\sum_{\Pi_M\subset\Pi_{M'}}
    (-v^2)^{|\Pi_M|}}.$

(2) $(w_0:\xi_{M'})=(-v)^{|\Pi|-|\Pi_{M'}|}\cdot (\xi_{M'}:\xi_{M'}).$

\end{proposition}

\begin{proof}

To prove (1), it suffices to prove the identity when $M'=G.$ We will
show first that $(\xi_G:w_0(M_1))=0,$ for all $\Pi_{M_1}\subsetneq
\Pi.$ We have $(\xi_G:w_0(M_1))=\displaystyle{\sum_{\Pi_M\subset\Pi}
v^{|\Pi|-|\Pi_M|}
(-v)^{|\Pi_M\vee\Pi_{M_1}|}}=(-1)^{|\Pi_{M_1}|}v^{|\Pi|+|\Pi_{M_1}|}\sum_{\Pi_M\subset\Pi}
x_{M,M_1},$ where we denoted $x_{M,M_1}=(-1)^{|\Pi_M|}
(v^{-2})^{|\Pi_{M}\cap \Pi_{M_1}|}.$ 

Let $\al$ be a root such that $\al\in\Pi\setminus\Pi_{M_1}.$ Then, for
every $\Pi_M\subset\Pi\setminus\{\al\},$ $x_{M,M_1}=-x_{M\cup\{\al\},M_1}.$ The
last sum can be written as $\displaystyle{\sum_{\Pi_M\subset\Pi\setminus\{\al\}}
}x_{M,M_1}+\displaystyle{\sum_{\Pi_M\subset\Pi\setminus\{\al\}}}
x_{M\cup\{\al\},M_1}=0.$

This implies that $(\xi_G:\xi_G)=(\xi_G:w_0)=\displaystyle{\sum_{\Pi_M\subset\Pi_{M'}}(-v^2)^{|\Pi_{M'}|-|\Pi_M|}}.$

\medskip

Formula (2) follows immediately from (1):

\noindent $(w_0:\xi_{M'})=\displaystyle{\sum_{\Pi_M\subset
    \Pi_{M'}}v^{|\Pi_{M'}|-|\Pi_M|}\cdot (-v)^{|\Pi|-|\Pi_M|}}=(-v)^{|\Pi|-|\Pi_M|}(\xi_{M'}:\xi_{M'}).$

\end{proof}

The basis elements in $\C Z_+$,
respectively $\C U_+$ are
obtained from those of $\C Z_-$, respectively $\C U_-$ by
multiplication by $w_0$ on the right. 
From proposition \ref{p:1} and in view of the algorithm, we can
determine the basis elements of $\C Z_-$ and $\C Z_+.$

\begin{corollary}
The bases are $\C Z_-=\{\xi_{M'}:\Pi_M'\subset \Pi\}$ and $\C
U_-=\{w_0(M'):\Pi_{M'}\subset \Pi\}.$ Moreover, after the sign
normalization, the polynomials are
\begin{equation}c_{M_1,M_2}=\left\{\begin{matrix} 0,&\text{ if } \Pi_{M_1}\not\subseteq
\Pi_{M_2},\\ v^{|\Pi_{M_2}|-|\Pi_{M_1}|}, &\text{ if }\Pi_{M_1}\subseteq\Pi_{M_2}.
\end{matrix}\right.\end{equation}

\end{corollary}

\begin{proof}
It remains to verify that
$\xi_G=w_0-\displaystyle{\sum_{\Pi_{M'}\subsetneq\Pi}(-v)^{|\Pi|-|\Pi_{M'}|}\xi_{M'}}.$
The right hand side equals
$RHS=w_0-\displaystyle{\sum_{\Pi_{M'}\subsetneq\Pi}\sum{\Pi_M\subset\Pi_{M'}}(-1)^{|\Pi|-|\Pi_{M'}|}
v^{|\Pi|-|\Pi_{M}|}w_0(M)}.$ We rewrite it as
$RHS=w_0-\displaystyle{\sum_{\Pi_M\subsetneq\Pi}v^{|\Pi|-|\Pi_M|}
  w_0(M)}\left(\displaystyle{-1+\sum_{\Pi_{M'}\supset\Pi_M}(-1)^{|\Pi|-|\Pi_{M'}|}}\right).$ Finally, $\displaystyle{ \sum_{\Pi_{M'}\supset\Pi_M}(-1)^{|\Pi|-|\Pi_{M'}|}}=(-1)^{|\Pi|-|\Pi_M|}\displaystyle{\sum_{S\subset\Pi\setminus\Pi_M}(-1)^{|S|}}=0.$  

\end{proof}

In terms of Kazhdan-Lusztig polynomials, this result is formulated as
follows:
\begin{equation}P_{M_1,M_2}(q)=\left\{\begin{matrix} 0,&\text{ if } \Pi_{M_1}\not\subseteq
\Pi_{M_2},\\ 1, &\text{ if }\Pi_{M_1}\subseteq \Pi_{M_2}.
\end{matrix}\right.\end{equation}

\subsection{Zelevinsky's example in $gl(4)$} This is one of the first examples
of nontrivial Kazhdan-Lusztig polynomials (see 11.4 in \cite{Ze3}). Consider $\chi=(2,0,0,-2)$
in $gl(4)$ for simple roots
$\Pi=\{\ep_1-\ep_2,\ep_2-\ep_3,\ep_3-\ep_4\}.$ 
All local systems are trivial. 
The list of orbits is:

\
\begin{tiny}
\noindent\begin{tabular}{|c|c|c|c|c|}
\hline
Dimension &$s$ &$G$-saturation &$\mathcal Z_-$
&$\mathcal U_-$\\
\hline
$0$ &$(0,0,0,0)$ &$(1^4)$ &$\frac v{1+v^2}[1]$ &$\frac v{1+v^2}[1]$\\
$2_a$ &$(1,-1,0,0)$ &$(211)$ &$[s_1]+v[1]$ &$[s_1]+\frac v{1+v^2}[1]$\\
$2_b$ &$(0,0,1,-1)$ &$(211)$ &$[s_3]+v[1]$ &$[s_3]+\frac v{1+v^2}[1]$\\
$3$   &$(1,-1,1,-1)$ &$(22)$ &$[s_1s_3]+v[s_1]+v[s_3]+v^2[1]$ &$[s_1s_3]$\\
$4$   &$(2,0,0,-2)$  &$(31)$ &$\frac v{1+v^2}[w_0]-v[s_1s_3]+\frac
    {v^3}{1+v^2}[1]$ &$\frac v{1+v^2}[w_0]$\\
\hline
\end{tabular}
\end{tiny}

\

The change of basis matrix from $\C Z_-$ to $\C U_-$ is:
$\left(\begin{matrix}
1&-v^2&-v^2&v+v^3&v^4\\
0& 1 &0 & -v &-v^2\\
0&0&1&-v&-v^2\\
0&0&0&1&v\\
0&0&0&0&1
\end{matrix}\right),$
and the matrix of Kazhdan-Lusztig polynomials is:

\begin{center}\(\left(
\begin{array}{l||cccc|c}
 &0&2_a&2_b&3&4\\
\hline\hline
0& 1 & 1 & 1 & 1+q & 1  \\
2_a& 0 & 1 & 0 & 1 & 1  \\
2_b& 0 & 0 & 1 & 1 & 1  \\
3& 0 & 0 & 0 & 1  & 1 \\
\hline
4& 0 & 0 & 0 & 0  & 1
\end{array}
\right)\).
\end{center}

The action of the involution $IM$ is:
\begin{equation}
IM(5)=0,\ IM(3)=3,\ IM(2_a)=2_b.
\end{equation}

\subsection{$\chi=(1,1)$ in ${sp}(4)$}

Consider $\chi=(1,1),$ the middle element of the nilpotent orbit
$(22)$ in ${sp}(4).$ There are $3$ orbits in $Orb_2(\chi),$ the open
orbit with two local systems. The first one listed below is the
trivial. Each Weyl group coset $W/W(\chi)$ is given by the action of a
representative element on $\chi.$

\
\begin{tiny}
\begin{tabular}{|c|c|c|c|}
\hline
Dimension &$s$   &$\mathcal Z_-$ &$\mathcal U_-$\\
\hline
$0$ &$(0,0)$   &$\frac v{1+v^2}[1,1]$ &$\frac
v{1+v^2}[1,1]$ \\
\hline
$2$ &$(0,1)$   &$[1,-1]+v[1,1]$&$[1,-1]+\frac
v{1+v^2}[1,1]$  \\
\hline
$3$ &$(1,1)$   &$\frac v{1+v^2}[-1,-1]-v[1,-1]-\frac
{v^2}{1+v^2}[1,1]$&$\frac v{1+v^2}[-1,-1]$  \\
\hline
 & &$[-1,1]+\frac v{1+v^2}[-1,-1]-\frac{v^2}{1+v^2}[1,1]$&$[-1,1]+\frac
v{1+v^2}[-1,-1]$ \\ 
\hline
\end{tabular}
\end{tiny}

\medskip

The change of basis matrix is 
$\left(
\begin{array}{cc|cc}
 1 & -v^2 & -v^3 & v \\
 0 & 1 & v & 0  \\
\hline
  0 & 0 & 1 & 0 \\
 0 & 0 & 0 & 1
\end{array}
\right),$ and the matrix of Kazhdan-Lusztig polynomials is 
$$\left(
\begin{array}{l||cc|cc}
 &0&2&3_{triv}&3_{sgn}\\
\hline\hline
0& 1 & 1 & 1 & q \\
2& 0 & 1 & 1 & 0  \\
\hline
3_{triv} & 0 & 0 & 1 & 0 \\
3_{sgn}& 0 & 0 & 0 & 1
\end{array}
\right).$$

The action of the involution $IM$ is 
\begin{equation}
IM(3)=1,\ IM(4)=2.
\end{equation}

\subsection{$\chi=(3,1,1)$ in  ${sp}(6)$}

The central character is
$\chi=(3,1,1)$, the middle element of the triangular nilpotent $(4,2)$
in $sp(6).$ There are
$10$ orbits in $Orb_2(\chi)$, two orbits (one of which is the open
orbit) with two local systems. We list the
parameterization of these orbits, the dimensions, the corresponding
Levi subalgebras and the basis elements $\mathcal Z_-$ and $\mathcal
U_-.$ The bases $\mathcal Z_+$ and $\mathcal U_+$ are obtained by
multiplication by $w_0.$

We encode the cosets $W/W(\chi)$ by the $W$ action on $(3,1,1).$

\

\begin{tiny}
\noindent\begin{tabular}{|c|c|c|}
\hline
Dim &$s$   &$\mathcal Z_-$ \\
\hline
\hline
$0$ &$(0,0,0)$  
&$\frac{1}{v+v^{-1}}[3,1,1]$\\
\hline
$2_a$ &$(1,-1,0)$   &$[1,3,1]+v[3,1,1]$
\\
\hline
$2_b$ &$(0,0,1)$ &$[3,1,-1]+v[3,1,1]$
 \\
\hline
$3_a$ &$(1,-1,1)$  
  &$[1,3,-1]+v[1,3,1]+v[3,1,-1]+v^2[3,1,1]$ \\
\hline

$3_{b,t}$ &$(0,1,1)$  
  &$\frac{1}{v+v^{-1}}[3,-1,-1]-v[3,1,-1]-\frac{v}{v+v^{-1}}[3,1,1]$
  \\

\hline

$3_{b,s}$   &
&$[3,-1,1]+\frac{1}{v+v^{-1}}[3,-1,-1]-\frac{v}{v+v^{-1}}[3,1,1]$
\\
\hline

$4_a$ &$(2,0,2)$  
&$[-1,1,3]+v[1,3,-1]+v[3,-1,1]+v^2[3,1,-1]$ \\
\hline

$4_b$ &$(3,1,0)$  
  &$[1,-3,-1]+v[1,1,3]+v[1,3,-1]+v^2[1,3,1]$ \\

\hline

$5_t$ &$(3,1,1)$  &$\frac 1{v+v^{-1}}[-3,-1,-1]-v[-1,1,3]-v[1,-3,-1]$\\&&$-\frac{v^2}{v+v^{-1}}[3,-1,-1]-v^2[1,1,3]-v^2[1,3,-1]$\\&&$-v^2[3,-1,1]-v^3[1,3,1]$  \\

\hline

$5_s$  & &$[-3,-1,1]+\frac{1}{v+v^{-1}}[-3,-1,-1]+v[-1,1,3]+v[1,3,1]$\\&&$-\frac{v^2}{v+v^{-1}}[3,-1,-1]+v^2[1,3,-1]+v^2[3,1,1]+v^3[3,1,-1]$  \\

\hline
\end{tabular}

\end{tiny}

\medskip

\begin{tiny}
\noindent\begin{tabular}{|c|c|c|}
\hline
$s$  &$\mathcal
U_-$ &$G$-saturation\\
\hline
\hline
$(0,0,0)$ &$\frac{1}{v+v^{-1}}[3,1,1]$ &$(1^6)$\\
\hline
$(1,-1,0)$ 
&$[1,3,1]+\frac{1}{v+v^{-1}}[3,1,1]$ &$(221^2)$\\
\hline
$(0,0,1)$ 
  &$[3,1,-1]+\frac{1}{v+v^{-1}} [3,1,1]$ &$(21^4)$\\
\hline
$(1,-1,1)$  &$[1,3,-1]$ &$(2,2,2)$\\
\hline

$(0,1,1)$ 
  &$\frac{1}{v+v^{-1}}[3,-1,-1]$ &$(221^2)$\\

\hline

 &$[3,-1,1]+\frac{1}{v+v^{-1}}[3,-1,-1]$ &\\
\hline

$(2,0,2)$  &$[-1,1,3]$ &$(33)$\\
\hline

$(3,1,0)$
&$[1,-3,-1]+v[1,1,3]+v^2[1,3,-1]+\frac{v^2}{v+v^{-1}}[1,3,1]$ &$(41^2)$\\

\hline

$(3,1,1)$   &$\frac 1{v+v^{-1}}[-3,-1,-1]$ &$(42)$\\

\hline

  &$[-3,-1,1]+\frac 1{v+v^{-1}}[-3,-1,-1]$ &\\

\hline
\end{tabular}

\end{tiny}

\medskip

The change of basis matrix from $\mathcal Z_-$ to $\mathcal U_-$  is:

\

\begin{tiny}$$\left(
\begin{array}{cccccccc|cc}
 1 & -v^2 & -v^2 & (v+v^3) & -v^3 & v & -(v^2+v^4) & -v^4 & -v^5 & v^3 \\
 0 & 1 & 0 & -v & 0 & 0 & v^2 & v^2 & v^3 & -v \\
 0 & 0 & 1 & -v & v & 0 & v^2 & v^2 & v^3 & 0 \\
 0 & 0 & 0 & 1 & 0 & 0 & -v & -v & -v^2 & 0 \\
 0 & 0 & 0 & 0 & 1 & 0 & v & 0 & v^2 & 0 \\
 0 & 0 & 0 & 0 & 0 & 1 & -v & 0 & 0 & v^2 \\
 0 & 0 & 0 & 0 & 0 & 0 & 1 & 0 & v & -v \\
 0 & 0 & 0 & 0 & 0 & 0 & 0 & 1 & v & 0 \\
\hline
 0 & 0 & 0 & 0 & 0 & 0 & 0 & 0 & 1 & 0 \\
 0 & 0 & 0 & 0 & 0 & 0 & 0 & 0 & 0 & 1
\end{array}
\right),$$\end{tiny}

and the matrix of Kazhdan-Lusztig polynomials is

$$\left(
\begin{array}{l||cccccccc|cc}
 &0&2_a&2_b&3_a&3_{b,t}&3_{b,s}&4_a&4_b&5_t&5_s\\
\hline\hline
0& 1 & 1 & 1 & 1+q & 1 & q & 1+q & q & 1 & q \\
2_a& 0 & 1 & 0 & 1 & 0 & 0 & 1 & 1 & 1 & q \\
2_b& 0 & 0 & 1 & 1 & 1 & 0 & 1 & 1 & 1 & 0 \\
3_a& 0 & 0 & 0 & 1 & 0 & 0 & 1 & 1 & 1 & 0 \\
3_{b,t}& 0 & 0 & 0 & 0 & 1 & 0 & 1 & 0 & 1 & 0 \\
3_{b,s}& 0 & 0 & 0 & 0 & 0 & 1 & 1 & 0 & 0 & 1 \\
4_a& 0 & 0 & 0 & 0 & 0 & 0 & 1 & 0 & 1 & 1 \\
4_b& 0 & 0 & 0 & 0 & 0 & 0 & 0 & 1 & 1 & 0 \\
\hline
5_t& 0 & 0 & 0 & 0 & 0 & 0 & 0 & 0 & 1 & 0 \\
5_s& 0 & 0 & 0 & 0 & 0 & 0 & 0 & 0 & 0 & 1
\end{array}
\right).$$
The action of the involution $IM$ is:
\begin{equation}
IM(1)=9,\ IM(2)=5,\ IM(3)=10,\ IM(4)=7,\ IM(6)=8.
\end{equation}

\subsection{$G_2$} 
There are five nilpotent orbits. The regular orbit is a particular case
of section \ref{sec:3.1}, and we will also ignore the trivial orbit.

\subsubsection{$G_2(a_1)$} The central character is
$\chi$, the middle element of the nilpotent $G_2(a_1).$ There are
$4$ orbits of $G(\chi)$ on $\mathfrak g_2(\chi),$ and it turns out they are distinguished by their $G$-saturations. They are:

\begin{tabular}{l|l}
Dimension &$G$-saturation\\ 
\hline
$0$       &$0$\\
$2$       &$A_1$\\
$3$       &$\wti A_1$\\
$4$       &$G_2(a_1)$. 
\end{tabular}

The closure ordering is $$0-2-3-4.$$
The stabilizer of a point in the dense orbit ($4$-dimensional) is
$S_3,$ but only $2$ local systems appear for the equal parameter case
(the extra local system is cuspidal).

\medskip

The matrix of Kazhdan-Lusztig polynomials is:

\medskip

$$\left(
\begin{array}{l||ccc|cc}
  &0&2&3&4_t&4_s\\
\hline\hline
0 &1 & 1 & q+1 & 1 & q  \\
2& 0 & 1 & 1 & 1 & 0 \\
3& 0 & 0 & 1 & 1 & 1 \\
\hline
4_t& 0 & 0 & 0 & 1 & 0 \\
4_s& 0 & 0 & 0 & 0 & 1 \\
\end{array}
\right)$$

In terms of the classical Langlands classification for the $p$-adic
group $\C G$ of type $G_2$ whose dual is $G$, the rows correspond to
the induced standard modules from: the Borel subgroup, the parabolic
of type $A_1$ short, the parabolic of type $A_1$ long, and two
discrete series (the first generic) respectively. We denote these induced modules by
$X(0)$ (this is the full unramified principal series), $X(A_1^s),$
$X(A_1^l),$ and $DS(g),$ $DS(ng)$ respectively. The columns
correspond to the Langlands quotients: $\overline X(0),$ $\overline
X(A_1^s),$ $\overline X(A_1^l),$ and $DS(g),$ $DS(ng)$
respectively. Therefore the character decompositions are:
\begin{align}
X(0)&=\overline X(0)+\overline X(A_1^s)+2\cdot\overline
X(A_1^l)+DS(g)+DS(ng);\\
X(A_1^s)&=\overline X(A_1^s)+\overline X(A_1^l)+DS(g);\\
X(A_1^l)&=\overline X(A_1^l)+DS(g)+DS(ng);\\ 
DS(g)&=DS(g);\\
DS(ng)&=DS(ng).
\end{align}
Finally,
\begin{equation}
IM(DS(g))=\overline X(0),\quad IM(DS(ng))=\overline X(A_1^s), \quad
IM(\overline X(A_1^l))=\overline X(A_1^l).
\end{equation}

\subsubsection{$A_1$ or $\wti A_1$} If $\chi$ is the middle element of
the nilpotent 
$A_1$ or $\wti A_1,$ then $Orb_2(\chi)$ has only two elements, the
zero orbit and the dense orbit of dimension one, and the
component groups are trivial. The matrix of Kazhdan-Lusztig
polynomials is, in both cases,
$$\left(
\begin{array}{l||c|c}
 &0&1\\
\hline\hline
0& 1 & 1 \\
\hline
1& 0 & 1 
\end{array}
\right).$$

\section{Polynomials for $F_4$}\label{sec:4}There are $16$
nilpotent orbits in $F_4.$ We compute the polynomials for the central
characters $\chi$ that are  middle elements of  nilpotent orbits. The
most interesting example is when $\chi$ is the middle element of the
nilpotent $F_4(a_3).$ This is the one with component group of
$S_4$. We first present this example in detail and then record the
other $15$ cases.

We also give the Iwahori-Matsumoto dual of the  tempered
modules. When we write $IM(\xi)=\xi',$ we mean that the $IM$ dual of
the simple module parameterized by $\xi$ is the simple module
parameterized by $\xi'.$

\smallskip

The simple roots we use for $F_4$ are $\alpha_1=(1,-1,-1,-1),$
$\alpha_2=(0,0,0,2),$  $\alpha_3=(0,0,1,-1),$ $\alpha_4=(0,1,-1,0).$
In these coordinates, the ``most interesting'' $\chi$ is $(3,1,1,1).$

The notation for orbits and local systems is as explained at the beginning of section \ref{sec:3}. The only change is that for the open orbit at $\chi=(3,1,1,1),$ the component group being $S_4,$ we use a subscript denoting the partition of $4$ which labels the corresponding irreducible representation of $S_4.$

\subsection{$\chi=(3,1,1,1)$ in $F_4$}\label{sec:3111} This is the middle element of the
nilpotent orbit $F_4(a_3)$. There are $12$ orbits and a total of $20$
local systems. (This was previously known by \cite{DLP}). The component group of the stabilizer of a point in the
open orbit is $S_4,$ so there are irreducible $5$ local systems, one
of which is cuspidal in the sense of Lusztig. We will not consider it
because it doesn't parameterize $mod_\chi(\bH)$.
 Therefore, our matrix has $19$ columns and
rows (corresponding to $12$ orbits), where the last $4$ correspond to
the open orbit. 

The list of orbits follows. For each orbit, we give a label which encodes
the dimension as well, the
semisimple element $s$ (from which the good parabolic  is constructed),
the $G$-saturation of the orbit, and the component group of the
stabilizer in $G(\chi)$ of a point in the orbit.

\begin{tabular}{|c|c|c|c|}
\hline
Dimension &$s$ &$G$-saturation &Components\\
\hline
$0$ &$(0,0,0,0)$ &$0$ &$1$\\
$4$ &$(0,0,0,1)$ &$A_1$ &$1$\\
$6$ &$(0,0,1,1)$ &$\wti A_1$ &$\bZ/2\bZ$\\
$7'$ &$(\frac 12,-\frac 12,\frac 12,\frac 32)$ &$A_1+\wti A_1$ &$1$\\
$7''$ &$(0,1,1,1)$ &$A_1+\wti A_1$ &$1$\\
$8'$ &$(1,-1,1,1)$ &$A_2$ &$\bZ/2\bZ$\\
$8''$ &$(2,0,0,2)$ &$\wti A_2$ &$1$\\
$9$   &$(1,0,1,2)$ &$\wti A_1+A_2$ &$1$\\
$10'$ &$(2,0,1,2)$ &$A_1+\wti A_2$ &$1$\\
$10''$ &$(2,1,1,2)$ &$B_2$ &$\bZ/2\bZ$\\
$11$  &$(\frac 52,\frac 12,\frac 32,\frac 32)$  &$C_3(a_1)$ &$\bZ/2\bZ$\\
$12$ &$(3,1,1,1)$ &$F_4(a_3)$ &$S_4$\\
\hline
\end{tabular}

The Kazhdan-Lusztig polynomials are in the following matrix. There are
$19$ columns, each corresponding to one of the local systems in the
table above. Due to the size of the matrix of polynomials, we break it
into three parts. There is a subtle issue of identifying the
three nontrivial local systems on the open orbit $12$. Note in the
matrix, that they are distinguished by their multiplicity in the first
row. In representation theoretic language, this means that the three
corresponding nongeneric discrete series are distinguished by their
multiplicity in the spherical principal series. Then to complete the
identification, we referred to \cite{Ci}.

\begin{tiny}
\noindent\(\text{Columns 1-8:}\)

\

\noindent\(\left(
\begin{array}{l||cccccccc}
 & 0 &4 &6_t &6_s &7_a &7_b &8_{a,t} &8_{a,s}\\
\hline\hline
 0 & 1 & 1+q & 1+q+q^2 & q+q^2 & 1+2 q+2 q^2+q^3 & 1+q & 1+q+q^2 & q \\
 4 & 0 & 1 & 1 & 0 & 1+q & 1 & 1+q & q \\
 6_t & 0 & 0 & 1 & 0 & 1 & 1 & 1 & 0 \\
 6_s  & 0 & 0 & 0 & 1 & 1 & 0 & 0 & 0 \\
 7_a  &0 & 0 & 0 & 0 & 1 & 0 & 1 & 0 \\
 7_b  &0 & 0 & 0 & 0 & 0 & 1 & 0 & 0 \\
 8_{a,t}  &0 & 0 & 0 & 0 & 0 & 0 & 1 & 0 \\
 8_{a,s}  &0 & 0 & 0 & 0 & 0 & 0 & 0 & 1 \\
 8_b  &0 & 0 & 0 & 0 & 0 & 0 & 0 & 0 \\
 9  & 0 & 0 & 0 & 0 & 0 & 0 & 0 & 0 \\
 10_a  & 0 & 0 & 0 & 0 & 0 & 0 & 0 & 0 \\
 10_{b,t} & 0 & 0 & 0 & 0 & 0 & 0 & 0 & 0 \\
 10_{b,s} & 0 & 0 & 0 & 0 & 0 & 0 & 0 & 0 \\
 11_t  &0 & 0 & 0 & 0 & 0 & 0 & 0 & 0 \\
 11_s  &0 & 0 & 0 & 0 & 0 & 0 & 0 & 0 \\
 12_{(4)}  &0 & 0 & 0 & 0 & 0 & 0 & 0 & 0 \\
 12_{(31)}  &0 & 0 & 0 & 0 & 0 & 0 & 0 & 0 \\
 12_{(22)}  &0 & 0 & 0 & 0 & 0 & 0 & 0 & 0 \\
 12_{(211)}  &0 & 0 & 0 & 0 & 0 & 0 & 0 & 0
\end{array}
\right)\)

\

\

\noindent\(\text{Columns 9-13:}\)

\

\noindent\(\left(
\begin{array}{l||ccccc}
  & 8_b  & 9  & 10_a  & 10_{b,t}  & 10_{b,s} \\
\hline\hline
 0  &1+q+q^2 & 1+q+q^2+q^3 & 1+q+2 q^2+q^3+q^4 & 1+q & q+2 q^2+q^3 \\
 4  &1 & 1+q+q^2 & 1+q+q^2 & 1+q & q+q^2 \\
 6_t  &1 & 1 & 1+q & 1 & q \\
 6_s  &1 & 0 & q & 0 & q \\
 7_a  &1 & 1 & 1+q & 1 & q \\
 7_b  &0 & 1 & 1+q & 1 & q \\
 8_{a,t}  &0 & 1 & 1 & 1 & 0 \\
 8_{a,s}  &0 & 1 & 0 & 1 & 0 \\
 8_b  &1 & 0 & 1 & 0 & 0 \\
 9  &0 & 1 & 1 & 1 & 0 \\
 10_a  &0 & 0 & 1 & 0 & 0 \\
 10_{b,t}  &0 & 0 & 0 & 1 & 0 \\
 10_{b,s}  &0 & 0 & 0 & 0 & 1 \\
 11_t  &0 & 0 & 0 & 0 & 0 \\
 11_s  &0 & 0 & 0 & 0 & 0 \\
 12_{(4)}  &0 & 0 & 0 & 0 & 0 \\
 12_{(31)}  &0 & 0 & 0 & 0 & 0 \\
 12_{(22)}  &0 & 0 & 0 & 0 & 0 \\
 12_{(211)}  &0 & 0 & 0 & 0 & 0
\end{array}
\right)\)

\noindent\(\\
\)

\

\noindent\(\text{Columns 14-19:}\)

\

\noindent\(\left(
\begin{array}{l||cc|cccc}
& 11_t  & 11_s  & 12_{(4)}  & 12_{(31)}  & 12_{(22)}  & 12_{(211)} \\
\hline\hline
 0  &1+2 q+2 q^2+q^3 & q+q^2+q^3+q^4 & 1 & q+q^2 & q+q^2+q^3 & q^3 \\
 4  &1+2 q+q^2 & q & 1 & q & q+q^2 & 0 \\
 6_t  &1+2 q & q & 1 & q & q+q^2 & 0 \\
 6_s  &q & q+q^2 & 0 & q & 0 & q^2 \\
 7_a  &1+2 q & q & 1 & q & q & 0 \\
 7_b  &1+q & q & 1 & q & q^2 & 0 \\
 8_{a,t}  &1+q & 0 & 1 & 0 & q & 0 \\
 8_{a,s}  &0 & 0 & 0 & 0 & 0 & 0 \\
 8_b  &1+q & 1+q & 1 & 1+q & q & q \\
 9  &1 & 0 & 1 & 0 & 0 & 0 \\
 10_a  &1 & 1 & 1 & 1 & 0 & 0 \\
 10_{b,t}  &1 & 0 & 1 & 0 & 1 & 0 \\
 10_{b,s}  &1 & 0 & 0 & 1 & 0 & 0 \\
 11_t  &1 & 0 & 1 & 1 & 1 & 0 \\
 11_s  &0 & 1 & 0 & 1 & 0 & 1 \\
\hline
 12_{(4)}  &0 & 0 & 1 & 0 & 0 & 0 \\
 12_{(31)}  &0 & 0 & 0 & 1 & 0 & 0 \\
 12_{(22)}  &0 & 0 & 0 & 0 & 1 & 0 \\
 12_{(211)}  &0 & 0 & 0 & 0 & 0 & 1
\end{array}
\right)\)

\end{tiny}

\

The Iwahori-Matsumoto involution gives:
\begin{equation}
IM(12_{(4)})=0,\quad IM(12_{(31)})=4,\quad IM(12_{(22)})=6_t,\quad IM(12_{(211)})=8_{a,s}.
\end{equation}

\

From the list of polynomials, we find that the closure ordering in $Orb_2(\chi)$ is as in figure \ref{f:3111}.

\begin{center}\begin{figure}[h]
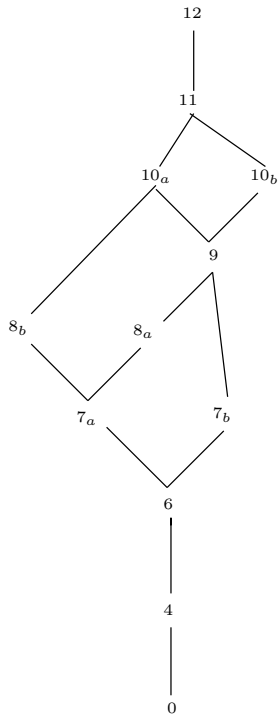
\label{f:3111}
\input F4_3111.pstex_t
\caption{The closure ordering for $Orb_2(\chi)$ in
    $F_4,$ where $\chi=(3,1,1,1).$}
\end{figure}
\end{center}

\subsection{The other $15$ cases} For each $\chi$ we give the list of
orbits and local systems, and the matrix of polynomials. 


\subsubsection{$\chi=(11,5,3,1), $ $\CO=F_4$} 
This is a particular case of the general case $\chi=2\check\rho$ in
section \ref{l:1}.

\subsubsection{$\chi=(7,3,1,1),$ $\CO=F_4(a_1)$} \ 

\begin{tiny}
\begin{tabular}{|c|c|c|c|}
\hline
Dimension &$s$ &$G$-saturation &Components\\
\hline
$0$ &$(0,0,0,0)$ &$0$ &$1$\\
$1$ &$(\frac 12,-\frac 12,-\frac 12,-\frac 12)$ &$A_1$ &$1$\\
$2_a$ &$(0,0,0,1)$ &$A_1$ &$1$\\
$2_b$ &$(0,1,-1,0)$ &$\wti A_1$ &$1$\\
$3_a$ &$(1,-1,-1,1)$ &$A_2$ &$1$\\
$3_b$ &$(\frac 12,\frac 12,-\frac 32,-\frac 12)$ &$A_1+\wti A_1$
&$1$\\
$3_c$ &$(0,1,-1,1)$ &$A_1+\wti A_1$ &$1$\\
$3_d$ &$(0,0,1,1)$ &$\wti A_1$ &$\bZ/2\bZ$\\
$4_a$ &$(0,2,0,2)$ &$\wti A_2$ &$1$\\
$4_b$ &$(0,3,0,1)$ &$C_2$ &$1$\\
$4_c$ &$(1,0,-2,1)$ &$A_2+\wti A_1$ &$1$\\
$4_d$ &$(2,-2,1,1)$ &$B_2$ &$\bZ/2\bZ$\\
$5_a$ &$(3,3,-3,1)$ &$B_3$ &$1$\\
$5_b$ &$(0,3,1,1)$ &$C_3(a_1)$ &$\bZ/2\bZ$\\
$5_c$ &$(\frac 92,\frac 12,-\frac 32,-\frac 72)$ &$C_3$ &$1$\\
$6$  &$(7,3,1,1)$ &$F_4(a_1)$ &$\bZ/2\bZ$\\
\hline
\end{tabular}
\end{tiny}
\begin{equation}
IM(6_t)=0,\quad IM(6_s)=2_a.
\end{equation}


\smallskip

\begin{tiny}
Columns $1-14$:

\noindent\(\left(
\begin{array}{l||cccccccccccccc}
& 0 &1 &2_a &2_b &3_a &3_b &3_c &3_{d,t} &3_{d,s} &4_a &4_b &4_c
  &4_{d,t} &4_{d,s} \\
\hline\hline
0 &1 & 1 & 1 & 1 & 1 & 1 & 1+q & 1 & q & 1+q & 1 & 1+q & 1 & q  \\ 
1 &0 & 1 & 0 & 0 & 1 & 1 & 0 & 0 & 0 & 0 & 0 & 1+q & 1 & q  \\ 
2_a &0 & 0 & 1 & 0 & 1 & 0 & 1 & 1 & 0 & 1& 1 & 1 & 1 & 0  \\
2_b &0 & 0 & 0 & 1 & 0 & 1 & 1 & 0 & 0 & 1 & 1 & 1 & 0 & 0  \\
3_a &0 & 0 & 0 & 0 & 1 & 0 & 0 & 0 & 0 & 0 & 0 & 1 & 1 & 0  \\
3_b &0 & 0 & 0 & 0 & 0 & 1 & 0 & 0 & 0 & 0& 0 & 1 & 0 & 0  \\
3_c &0 & 0 & 0 & 0 & 0 & 0 & 1 & 0 & 0 & 1 & 1 & 1 & 0 & 0  \\
3_{d,t} &0 & 0 & 0 & 0 & 0 & 0 & 0 & 1 & 0 & 1& 0 & 0 & 1 & 0  \\
3_{d,s} & 0 & 0 & 0 & 0 & 0 & 0 & 0 & 0 & 1 & 1& 0 & 0 & 0 & 1  \\
4_a& 0 & 0 & 0 & 0 & 0 & 0 & 0 & 0 & 0 & 1 & 0 & 0 & 0 & 0  \\
4_b& 0 & 0 & 0 & 0 & 0 & 0 & 0 & 0 & 0 & 0 & 1 & 0 & 0 & 0  \\
4_c& 0 & 0 & 0 & 0 & 0 & 0 & 0 & 0 & 0 & 0& 0 & 1 & 0 & 0  \\
4_{d,t}& 0 & 0 & 0 & 0 & 0 & 0 & 0 & 0 & 0 & 0 & 0 & 0 & 1 & 0  \\
4_{d,s}& 0 & 0 & 0 & 0 & 0 & 0 & 0 & 0 & 0 & 0 & 0 & 0 & 0 & 1  \\
5_a& 0 & 0 & 0 & 0 & 0 & 0 & 0 & 0 & 0 & 0 & 0 & 0 & 0 & 0  \\
 5_{b,t} &0 & 0 & 0 & 0 & 0 & 0 & 0 & 0 & 0 & 0& 0 & 0 & 0 & 0  \\
5_{b,s}& 0 & 0 & 0 & 0 & 0 & 0 & 0 & 0 & 0 & 0& 0 & 0 & 0 & 0  \\
5_c& 0 & 0 & 0 & 0 & 0 & 0 & 0 & 0 & 0 & 0& 0 & 0 & 0 & 0   \\
\hline
6_t& 0 & 0 & 0 & 0 & 0 & 0 & 0 & 0 & 0 & 0& 0 & 0 & 0 & 0  \\
6_s& 0 & 0 & 0 & 0 & 0 & 0 & 0 & 0 & 0 & 0& 0 & 0 & 0 & 0 
\end{array}
\right)\)
\end{tiny}

\smallskip

\begin{tiny}
Columns $15-20$:

\noindent\(\left(
\begin{array}{l||cccc|cc}
 &5_a &5_{b,t} &5_{b,s} &5_c &6_t &6_s\\
\hline\hline
0 
& 1 & 1  & q & 1+q & 1 & q \\ 
1 
 & 1 & 0 &  0 & 1+q & 1 & q \\ 
2_a 
 & 1 & 1 & 0 & 1 & 1 & 0 \\
2_b 
 & 1 & 1 & q & 1 & 1 & q \\
3_a 
& 1 & 0 & 0 & 1 & 1 & 0 \\
3_b 
& 1 & 0 & 0 & 1 & 1 & q \\
3_c 
& 1 & 1 & 0 & 1 & 1 & 0 \\
3_{d,t} 
 & 0 & 1 & 0 & 1 & 1 & 0 \\
3_{d,s} 
& 0 & 0 & 1 & 1 & 0 & 1 \\
4_a
& 0 & 1 & 1 & 1 & 1 & 1 \\
4_b
& 1 & 1 & 0 & 0 & 1 & 0 \\
4_c
& 1 & 0 & 0 & 1 & 1 & 0 \\
4_{d,t}
 & 0 & 0 & 0 & 1 & 1 & 0 \\
4_{d,s}
& 0 & 0 & 0 & 1 & 0 & 1 \\
5_a
& 1 & 0 & 0 & 0 & 1 & 0 \\
 5_{b,t} 
& 0 & 1 & 0 & 0 & 1 & 0 \\
5_{b,s}
& 0 & 0 & 1 & 0 & 0 & 1 \\
5_c
& 0 & 0 & 0 & 1 & 1 & 1 \\
\hline
6_t
& 0 & 0 & 0 & 0 & 1 & 0 \\
6_s
& 0 & 0 & 0 & 0 & 0 & 1
\end{array}
\right)\)
\end{tiny}

\subsubsection{$\chi=(5,3,1,1),$ $\CO=F_4(a_2)$}  \ 

\begin{tiny}
\begin{tabular}{|c|c|c|c|}
\hline
Dimension &$s$ &$G$-saturation &Components\\
\hline
$0$ &$(0,0,0,0)$ &$0$ &$1$\\
$2$ &$(0,1,-1,0)$ &$\wti A_1$ &$1$\\
$3$ &$(0,0,0,1)$ &$A_1$ &$1$\\
$4_a$ &$(0,1,-1,1)$ &$A_1+\wti A_1$ &$1$\\
$4_b$ &$(0,0,1,1)$ &$\wti A_1$ &$\bZ/2\bZ$\\
$5_a$ &$(0,2,0,2)$ &$\wti A_2$ &$1$\\
$5_b$ &$(\frac 12,-\frac 12,\frac 12,\frac 32)$ &$A_1+\wti A_1$ &$1$\\
$5_c$ &$(0,3,0,1)$ &$C_2$ &$1$\\
$6_a$ &$(\frac 12,\frac 32,-\frac 12,\frac 52)$ &$A_1+\wti A_2$ &$1$\\
$6_b$ &$(1,-1,1,1)$ &$A_2$ &$\bZ/2\bZ$\\
$6_c$ &$(0,3,1,1)$ &$C_3(a_1)$ &$\bZ/2\bZ$\\
$7_a$ &$(3,3,1,3)$ &$B_3$ &$1$\\
$7_b$ &$(\frac 92,\frac 72,\frac 12,\frac 32)$ &$C_3$ &$1$\\
$8$ &$(5,3,1,1)$ &$F_4(a_2)$ &$\bZ/2\bZ$\\
\hline
\end{tabular}
\end{tiny}
\begin{equation}
IM(8_t)=0,\quad IM(8_s)=4_{b,s}.
\end{equation}

\


\begin{tiny}
Columns $1-9$

\noindent\(\left(
\begin{array}{l||ccccccccc}
& 0&2&3&4_a&4_{b,t}&4_{b,s}&5_a&5_b&5_c\\
\hline\hline
0& 1 & 1 & 1+q & 1+q & 1+q & q & 1+2 q+q^2 & 1+q & 1+q\\

2& 0 & 1 & 0 & 1 & 0 & 0 & 1+q & 0 & 1+q\\

3& 0 & 0 & 1 & 1 & 1 & 0 & 1 & 1 & 1\\

4_a& 0 & 0 & 0 & 1 & 0 & 0 & 1 & 0 & 1\\

4_{b,t}& 0 & 0 & 0 & 0 & 1 & 0 & 1 & 1 & 0\\

4_{b,s}& 0 & 0 & 0 & 0 & 0 & 1 & 1 & 1 & 0 \\

5_a& 0 & 0 & 0 & 0 & 0 & 0 & 1 & 0 & 0\\

5_b& 0 & 0 & 0 & 0 & 0 & 0 & 0 & 1 & 0 \\

5_c& 0 & 0 & 0 & 0 & 0 & 0 & 0 & 0 & 1\\

6_a& 0 & 0 & 0 & 0 & 0 & 0 & 0 & 0 & 0\\

6_{b,t}& 0 & 0 & 0 & 0 & 0 & 0 & 0 & 0 &0 \\

6_{b,s}& 0 & 0 & 0 & 0 & 0 & 0 & 0 & 0 &0\\

6_{c,t}& 0 & 0 & 0 & 0 & 0 & 0 & 0 & 0 &0\\

6_{c,s}& 0 & 0 & 0 & 0 & 0 & 0 & 0 & 0 &0\\

7_a& 0 & 0 & 0 & 0 & 0 & 0 & 0 & 0 & 0 \\

7_b& 0 & 0 & 0 & 0 & 0 & 0 & 0 & 0 & 0\\

\hline
8_t& 0 & 0 & 0 & 0 & 0 & 0 & 0 & 0 & 0\\

8_s& 0 & 0 & 0 & 0 & 0 & 0 & 0 & 0 & 0\\

\end{array}
\right)\)
\end{tiny}

\

\begin{tiny}
Columns $10-18$:

\noindent\(\left(
\begin{array}{l||ccccccc|cc}
&6_a&6_{b,t}&6_{b,s}&6_{c,t}&6_{c,s}&7_a&7_b&8_t&8_s\\
\hline\hline
0 & 1+q & 1 & q & 1+q & q & 1+2 q+q^2 & 1+q & 1 & q \\
2 & 1 & 0 & 0 & 1+q & q & 1+q & 1+q & 1 & q \\
3 & 1& 1 & q & 1 & 0 & 1+2 q & 1 & 1 & q  \\
4_a & 1 & 0 & 0 & 1 & 0 & 1+q & 1 & 1 & q \\
4_{b,t} & 1& 1 & 0 & 1 & 0 & 1+q & 1 & 1 & 0  \\
4_{b,s} & 1& 0 & 0 & 0 & 1 & 0 & 1 & 0 & 0  \\
5_a & 1& 0 & 0 & 1 & 1 & 1 & 2 & 1 & 0  \\
5_b & 1 & 1 & 0 & 0 & 0 & 1 & 1 & 1 & 0 \\
5_c & 0& 0 & 0 & 1 & 0 & 1 & 1 & 1 & q  \\
6_a & 1 & 0 & 0 & 0 & 0 & 1 & 1 & 1 & 0 \\
6_{b,t} & 0& 1 & 0 & 0 & 0 & 1 & 0 & 1 & 0  \\
6_{b,s} & 0& 0 & 1 & 0 & 0 & 1 & 0 & 0 & 1  \\
6_{c,t} & 0& 0 & 0 & 1 & 0 & 1 & 1 & 1 & 0  \\
6_{c,s}& 0& 0 & 0 & 0 & 1 & 0 & 1 & 0 & 0  \\
7_a & 0 & 0 & 0 & 0 & 0 & 1 & 0 & 1 & 1 \\
7_b & 0& 0 & 0 & 0 & 0 & 0 & 1 & 1 & 0
 \\
\hline
8_t & 0& 0 & 0 & 0 & 0 & 0 & 0 & 1 & 0  \\
8_s & 0& 0 & 0 & 0 & 0 & 0 & 0 & 0 & 1
\end{array}
\right)\)
\end{tiny}

\subsubsection{$\chi=(5,1,1,1),$ $\CO=B_3$} \ 

\noindent\begin{tiny}
\begin{tabular}{|c|c|c|c|}
\hline
Dim. &$s$ &$G$-sat. &Comp.\\
\hline
$0$ &$(0,0,0,0)$ &$0$ &$1$\\
$1$ &$(\frac 12,-\frac 12,-\frac 12,-\frac 12)$ &$A_1$ &$1$\\
$3$ &$(0,0,0,1)$ &$A_1$ &$1$\\
$4$ &$(1,-1,-1,1)$ &$A_2$ &$1$\\
$5$ &$(0,0,1,1)$ &$\wti A_1$ &$\bZ/2\bZ$\\
$6_a$ &$(2,-2,1,1)$ &$B_2$ &$\bZ/2\bZ$\\
$6_b$ &$(0,1,1,1)$ &$A_1+\wti A_1$ &$1$\\
$7$ &$(5,1,1,1)$ &$B_3$ &$1$\\
\hline
\end{tabular}
\end{tiny}
\begin{tiny}
\noindent\(\left(
\begin{array}{l||ccccccccc|c}
& 0& 1& 3& 4& 5_t& 5_s& 6_{a,t}& 6_{a,s}& 6_b&7\\
\hline\hline
0& 1 & 1 & 1 & 1 & 1+q^2 & q & 1+q^2 & q & 1 & 1 \\
1& 0 & 1 & 0 & 1 & 0 & 0 & 1+q^2 & q & 0 & 1 \\
3& 0 & 0 & 1 & 1 & 1 & 0 & 1 & 0 & 1 & 1 \\
4& 0 & 0 & 0 & 1 & 0 & 0 & 1 & 0 & 0 & 1 \\
5_t& 0 & 0 & 0 & 0 & 1 & 0 & 1 & 0 & 1 & 1 \\
5_s& 0 & 0 & 0 & 0 & 0 & 1 & 0 & 1 & 0 & 0 \\
6_{a,t}& 0 & 0 & 0 & 0 & 0 & 0 & 1 & 0 & 0 & 1 \\
6_{a,s}& 0 & 0 & 0 & 0 & 0 & 0 & 0 & 1 & 0 & 0 \\
6_b& 0 & 0 & 0 & 0 & 0 & 0 & 0 & 0 & 1 & 1 \\
\hline
7& 0 & 0 & 0 & 0 & 0 & 0 & 0 & 0 & 0 & 1
\end{array}
\right)\)
\end{tiny}
\begin{equation}
IM(7)=0.
\end{equation}

\subsubsection{$\chi=(5,3,1,0),$ $\CO=C_3$}  \ 

\begin{tiny}
\begin{tabular}{|c|c|c|c|}
\hline
Dim. &$s$ &$G$-sat. &Comp.\\
\hline
$0$ &$(0,0,0,0)$ &$0$ &$1$\\
$1_a$ &$(1,-1,0,0)$ &$\wti A_1$ &$1$\\
$1_b$ &$(0,1,-1,0)$ &$\wti A_1$ &$1$\\
$1_c$ &$(0,0,1,0)$ &$A_1$ &$1$\\
$2_a$ &$(2,0,-2,0)$ &$\wti A_2$ &$1$\\
$2_b$ &$(1,-1,1,0)$ &$A_1+\wti A_1$ &$1$\\
$2_c$ &$(0,3,1,0)$ &$B_2$ &$1$\\
$3$ &$(5,3,1,0)$ &$C_3$ &$1$\\
\hline
\end{tabular}
\end{tiny}
\begin{tiny}
\noindent\(\left(
\begin{array}{l||ccccccc|c}
&0&1_a&1_b&1_c&2_a&2_b&2_c&3\\
\hline\hline
0& 1 & 1 & 1 & 1 & 1 & 1 & 1 & 1 \\
1_a& 0 & 1 & 0 & 0 & 1 & 1 & 0 & 1 \\
1_b& 0 & 0 & 1 & 0 & 1 & 0 & 1 & 1 \\
1_c& 0 & 0 & 0 & 1 & 0 & 1 & 1 & 1 \\
2_a& 0 & 0 & 0 & 0 & 1 & 0 & 0 & 1 \\
2_b& 0 & 0 & 0 & 0 & 0 & 1 & 0 & 1 \\
2_c& 0 & 0 & 0 & 0 & 0 & 0 & 1 & 1 \\
\hline
3& 0 & 0 & 0 & 0 & 0 & 0 & 0 & 1
\end{array}
\right)\)
\end{tiny}
\begin{equation}
IM(3)=0.
\end{equation}

\subsubsection{$\chi=(3,1,1,0),$ $\CO=C_3(a_1)$} \

\noindent\begin{tiny}
\begin{tabular}{|c|c|c|c|}
\hline
Dim. &$s$ &$G$-sat. &Comp.\\
\hline
$0$ &$(0,0,0,0)$ &$0$ &$1$\\
$2_a$ &$(1,-1,0,0)$ &$\wti A_1$ &$1$\\
$2_b$ &$(0,0,1,0)$ &$A_1$ &$1$\\
$3_a$ &$(1,-1,1,0)$ &$A_1+\wti A_1$ &$1$\\
$3_b$ &$(0,1,1,0)$ &$\wti A_1$ &$\bZ/2\bZ$\\
$4_a$ &$(2,0,2,0)$ &$\wti A_2$ &$1$\\
$4_b$ &$(3,0,1,0)$ &$B_2$ &$1$\\
$5$ &$(3,1,1,0)$ &$C_3(a_1)$  &$\bZ/2\bZ$\\
\hline
\end{tabular}
\end{tiny}
\begin{tiny}
\noindent\(\left(
\begin{array}{l||cccccccc|cc}
 &0&2_a&2_b&3_a&3_{b,t}&3_{b,s}&4_a&4_b&5_t&5_s\\
\hline\hline
0& 1 & 1 & 1 & 1+q & 1 & q & 1+q & 1 & 1 & q \\
2_a& 0 & 1 & 0 & 1 & 0 & 0 & 1 & 1 & 1 & q \\
2_b& 0 & 0 & 1 & 1 & 1 & 0 & 1 & 1 & 1 & 0 \\
3_a& 0 & 0 & 0 & 1 & 0 & 0 & 1 & 1 & 1 & 0 \\
3_{b,t}& 0 & 0 & 0 & 0 & 1 & 0 & 1 & 0 & 1 & 0 \\
3_{b,s}& 0 & 0 & 0 & 0 & 0 & 1 & 1 & 0 & 0 & 1 \\
4_a& 0 & 0 & 0 & 0 & 0 & 0 & 1 & 0 & 1 & 1 \\
4_b& 0 & 0 & 0 & 0 & 0 & 0 & 0 & 1 & 1 & 0 \\
\hline
5_t &0 & 0 & 0 & 0 & 0 & 0 & 0 & 0 & 1 & 0 \\
5_s& 0 & 0 & 0 & 0 & 0 & 0 & 0 & 0 & 0 & 1
\end{array}
\right)\)
\end{tiny}
\begin{equation}
IM(5_t)=0,\quad IM(5_s)=2_b.
\end{equation}

\subsubsection{$\chi=(\frac{5}{2},\frac{3}{2},\frac{1}{2},\frac{1}{2}),$
  $\CO=A_1+\wti A_2$}  \ 

\begin{tiny}
\begin{tabular}{|c|c|c|c|}
\hline
Dim. &$s$ &$G$-sat. &Comp.\\
\hline

$0$ &$(0,0,0,0)$ &$0$ &$1$\\
$1$ &$(\frac 12,-\frac 12,\frac 12,\frac 12)$ &$A_1$ &$1$\\
$3$ &$(0,1,0,1)$ &$\wti A_1$ &$1$\\
$4_a$ &$(\frac 12,\frac 12,\frac 12,\frac 32)$ &$A_1+\wti A_1$ &$1$\\
$4_b$ &$(2,2,0,0)$ &$\wti A_2$ &$1$\\
$5$ &$(\frac 52,\frac 32,\frac 12,\frac 12)$ &$A_1+\wti A_2$ &$1$\\
\hline
\end{tabular}
\end{tiny}
\begin{tiny}
\noindent\(\left(
\begin{array}{l||ccccc|c}
 &0&1&3&4_a&4_b&5\\
\hline\hline
0& 1 & 1 & 1+q & 1+q & 1 & 1 \\
1& 0 & 1 & 0 & 1+q & 0 & 1 \\
3& 0 & 0 & 1 & 1 & 1 & 1 \\
4_a& 0 & 0 & 0 & 1 & 0 & 1 \\
4_b& 0 & 0 & 0 & 0 & 1 & 1 \\
\hline
5& 0 & 0 & 0 & 0 & 0 & 1
\end{array}
\right)\)
\end{tiny}
\begin{equation}
IM(5)=0.
\end{equation}

\subsubsection{$\chi=(3,1,0,0),$ $\CO=B_2$} \

\begin{tiny}
\begin{tabular}{|c|c|c|c|}
\hline
Dim. &$s$ &$G$-sat. &Comp.\\
\hline

$0$ &$(0,0,0,0)$ &$0$ &$1$\\
$1$ &$(0,1,0,0)$ &$A_1$ &$1$\\
$4$ &$(\frac 12,-\frac 12,-\frac 12,-\frac 12)$ &$A_1$ &$1$\\
$5_a$ &$(1,1,-1,-1)$ &$A_2$ &$1$\\
$5_b$ &$(1,-1,0,0)$ &$\wti A_1$ &$\bZ/2\bZ$\\
$6$   &$(3,1,0,0)$ &$B_2$ &$\bZ/2\bZ$\\
\hline
\end{tabular}
\end{tiny}
\begin{tiny}
\noindent\(\left(
\begin{array}{l||cccccc|cc}
 &0&1&4&5_a&5_{b,t}&5_{b,s}&6_t&6_s\\
\hline\hline
0& 1 & 1 & 1 & 1 & 1 & q^2 & 1 & q^2 \\
1& 0 & 1 & 0 & 1 & 0 & 0 & 1 & q^2 \\
4& 0 & 0 & 1 & 1 & 1 & 0 & 1 & 0 \\
5_a& 0 & 0 & 0 & 1 & 0 & 0 & 1 & 0 \\
5_{b,t}& 0 & 0 & 0 & 0 & 1 & 0 & 1 & 0 \\
5_{b,s}& 0 & 0 & 0 & 0 & 0 & 1 & 0 & 1 \\
\hline
6_t& 0 & 0 & 0 & 0 & 0 & 0 & 1 & 0 \\
6_s& 0 & 0 & 0 & 0 & 0 & 0 & 0 & 1
\end{array}
\right)\)
\end{tiny}
\begin{equation}
IM(6_t)=0,\quad IM(6_s)=4.
\end{equation}

\subsubsection{$\chi=(2,1,1,0),$ $\CO=A_2+\wti A_1$}\

\noindent\begin{tiny}
\begin{tabular}{|c|c|c|c|}
\hline
Dim. &$s$ &$G$-sat. &Comp.\\
\hline

$0$ &$(0,0,0,0)$ &$0$ &$1$\\
$4$ &$(0,0,1,0)$ &$A_1$ &$1$\\
$5$ &$(0,1,1,0)$ &$\wti A_1$ &$\bZ/2\bZ$\\
$7$ &$(\frac 12,\frac 12,\frac 32,-\frac 12)$ &$A_1+\wti A_1$ &$1$\\
$8$ &$(1,-1,1,1)$ &$A_2$ &$\bZ/2\bZ$\\
$9$ &$(2,1,1,0)$ &$A_2+\wti A_1$ &$1$\\
\hline
\end{tabular}
\end{tiny}
\begin{tiny}
\noindent\(\left(
\begin{array}{l||ccccccc|c}
 &0&4&5_t&5_s&7&8_t&8_s&9\\
\hline\hline
0& 1 & 1+q & 1+q+q^2 & q & 1+q & 1+q+q^2 & q & 1 \\
4& 0 & 1 & 1 & 0 & 1 & 1+q & q & 1 \\
5_t& 0 & 0 & 1 & 0 & 1 & 1+q & 0 & 1 \\
5_s& 0 & 0 & 0 & 1 & 1 & 0 & 0 & 0 \\
7& 0 & 0 & 0 & 0 & 1 & 1 & 0 & 1 \\
8_t& 0 & 0 & 0 & 0 & 0 & 1 & 0 & 1 \\
8_s& 0 & 0 & 0 & 0 & 0 & 0 & 1 & 0 \\
\hline
9& 0 & 0 & 0 & 0 & 0 & 0 & 0 & 1
\end{array}
\right)\)
\end{tiny}
\begin{equation}
IM(9)=0.
\end{equation}

\subsubsection{$\chi=(2,0,0,0),$ $\CO=A_2$} \ 

\noindent\begin{tiny}
\begin{tabular}{|c|c|c|c|}
\hline
Dim. &$s$ &$G$-sat. &Comp.\\
\hline

$0$ &$(0,0,0,0)$ &$0$ &$1$\\
$7$ &$(\frac 12,-\frac 12,-\frac 12,-\frac 12)$ &$A_1$ &$1$\\
$10$ &$(1,-1,0,0)$ &$\wti A_1$ &$\bZ/2\bZ$\\
$13$ &$(\frac 32,-\frac 12,-\frac 12,-\frac 12)$ &$A_1+\wti A_1$ &$1$\\
$14$ &$(2,0,0,0)$ &$A_2$ &$\bZ/2\bZ$\\
\hline
\end{tabular}
\end{tiny}
\begin{tiny}
\noindent\(\left(
\begin{array}{l||ccccc|cc}
 &0&7&10_t&10_s&13&14_t&14_s\\
\hline\hline
0& 1 & 1+q^3 & 1+q^3 & q^2 & 1+q^3 & 1 & q^3 \\
7& 0 & 1 & 1 & 0 & 1 & 1 & q^3 \\
10_t& 0 & 0 & 1 & 0 & 1 & 1 & 0 \\
10_s &0 & 0 & 0 & 1 & q & 0 & 0 \\
13& 0 & 0 & 0 & 0 & 1 & 1 & 0 \\
\hline
14_t& 0 & 0 & 0 & 0 & 0 & 1 & 0 \\
14_s& 0 & 0 & 0 & 0 & 0 & 0 & 1
\end{array}
\right)\)
\end{tiny}
\begin{equation}
IM(14_t)=0,\quad IM(14_s)=10_s.
\end{equation}

\subsubsection{$\chi=(2,2,0,0),$ $\CO=\wti A_2$} \

\begin{tiny}
\begin{tabular}{|c|c|c|c|}
\hline
Dim. &$s$ &$G$-sat. &Comp.\\
\hline

$0$ &$(0,0,0,0)$ &$0$ &$1$\\
$7$ &$(0,1,-1,0)$ &$\wti A_1$ &$1$\\
$8$ &$(2,2,0,0)$ &$\wti A_2$ &$1$\\
\hline
\end{tabular}
\end{tiny}
\begin{tiny}
\noindent\(\left(
\begin{array}{l||cc|c}
 &0&7&8\\
\hline\hline
0& 1 & 1+q^3 & 1 \\
7& 0 & 1 & 1 \\
\hline
8& 0 & 0 & 1
\end{array}
\right)\)
\end{tiny}
\begin{equation}
IM(8)=0.
\end{equation}

\subsubsection{$\chi=(\frac{3}{2},\frac{1}{2},\frac{1}{2},\frac{1}{2}),$
  $\CO=A_1+\wti A_1$}\ 

\begin{tiny}
\begin{tabular}{|c|c|c|c|}
\hline
Dim. &$s$ &$G$-sat. &Comp.\\
\hline

$0$ &$(0,0,0,0)$ &$0$ &$1$\\
$3$ &$(\frac 12,-\frac 12,\frac 12,\frac 12)$ &$A_1$ &$1$\\
$5$ &$(1,0,0,1)$ &$\wti A_1$ &$\bZ/2\bZ$\\
$6$ &$(\frac 32,\frac 12,\frac 12,\frac 12)$ &$A_1+\wti A_1$ &$1$\\
\hline
\end{tabular}
\end{tiny}
\begin{tiny}
\noindent\(\left(
\begin{array}{l||cccc|c}
 &0&3&5_t&5_s&6\\
\hline\hline
0& 1 & 1 & 1+q^2 & q & 1 \\
3& 0 & 1 & 1 & 0 & 1 \\
5_t& 0 & 0 & 1 & 0 & 1 \\
5_s& 0 & 0 & 0 & 1 & 0 \\
\hline
6& 0 & 0 & 0 & 0 & 1
\end{array}
\right)\)
\end{tiny}
\begin{equation}
IM(6)=0.
\end{equation}

\subsubsection{$\chi=(1,1,0,0),$ $\CO=\wti A_1$} \

\begin{tiny}
\begin{tabular}{|c|c|c|c|}
\hline
Dim. &$s$ &$G$-sat. &Comp.\\
\hline

$0$ &$(0,0,0,0)$ &$0$ &$1$\\
$6$ &$(0,1,0,0)$ &$A_1$ &$1$\\
$7$ &$(1,1,0,0)$ &$\wti A_1$ &$\bZ/2\bZ$\\
\hline
\end{tabular}
\end{tiny}
\begin{tiny}
\noindent\(\left(
\begin{array}{l||cc|cc}
 &0&6&7_t&7_s\\
\hline\hline
0& 1 & 1 & 1 & q^3 \\
6& 0 & 1 & 1 & 0 \\
\hline
7_t& 0 & 0 & 1 & 0 \\
7_s& 0 & 0 & 0 & 1
\end{array}
\right)\)
\end{tiny}
\begin{equation}
IM(7_t)=0,\quad IM(7_s)=6.
\end{equation}

\subsubsection{$\chi=(1,0,0,0),$ $\CO=A_1$}\

\begin{tiny}
\begin{tabular}{|c|c|c|c|}
\hline
Dim. &$s$ &$G$-sat. &Comp.\\
\hline

$0$ &$(0,0,0,0)$ &$0$ &$1$\\
$1$ &$(1,0,0,0)$ &$A_1$ &$1$\\
\hline
\end{tabular}
\end{tiny}
\begin{tiny}
\noindent\(\left(
\begin{array}{l||c|c}
 &0&1\\
\hline\hline
0& 1 & 1 \\
\hline
1& 0 & 1
\end{array}
\right)\)
\end{tiny}
\begin{equation}
IM(1)=0.
\end{equation}

\subsubsection{$\chi=(0,0,0,0),$ $\CO=0$}
 There is only the trivial orbit, and only one polynomial equal to $1$.


\begin{thebibliography}{20}

\bibitem{BM}
D.~Barbasch, A.~Moy,
\emph{Whittaker models with an Iwahori fixed vector},  Representation theory and analysis on homogeneous spaces (New Brunswick, NJ, 1993),  101--105,
Contemp. Math., 177, Amer. Math. Soc., Providence, RI, 1994. 

\bibitem{Ca}
R.~Carter,
\emph{Finite groups of Lie type}, Wiley-Interscience, New York, 1985.

\bibitem{CG} 
N.~Chriss, V.~Ginzburg,
\emph{Representation theory and complex geometry}, Birkh\"auser Boston, Inc., Boston, MA, 1997.

\bibitem{Ci}
D.~Ciubotaru,
\emph{The $I$-spherical unitary dual of split $p$-adic $F_4$},
Represent. Theory 9(2005), 94--137.

\bibitem{DLP}
C.~De Concini, G.~Lusztig, C.~Procesi,
\emph{Homology of the zero-set of a nilpotent vector field on a flag manifold},
J. Amer. Math. Soc. 1 (1988), no. 1, 15--34. 


\bibitem{EM}
S.~Evens, I.~Mirkovi\'c
\emph{Fourier transform and the Iwahori-Matsumoto involution},  Duke Math. J.  86  (1997),  no. 3, 435--464.


\bibitem{Ka} 
N.~Kawanaka,
\emph{Orbits and stabilizers of nilpotent elements of a graded semisimple Lie algebra},  J. Fac. Sci. Univ. Tokyo Sect. IA Math.  34  (1987),  no. 3, 573--597.

\bibitem{KL}
D.~Kazhdan, G.~Lusztig,
\emph{ Proof of the Deligne-Langlands conjecture for Hecke algebras},
Inv Math., vol 87, 1987, 153-215.



\bibitem{L1}
G.~Lusztig,
\emph{Study of perverse sheaves arising from graded Lie algebras},
Adv. Math. 112 (1995), 147--217.

\bibitem{L2}
G.~Lusztig,
\emph{Graded Lie algebras and intersection cohomology}, preprint.


\bibitem{Lu3}
G.~Lusztig, 
\emph{Affine Hecke algebras and their graded version},
Jour. AMS, vol 2, 1989, 599-635.

\bibitem{Lu4}
G.~Lusztig,
\emph{Cuspidal local systems and graded algebras II},
Representations of groups (Banff, AB, 1994), Amer. Math. Soc.,
Providence, 1995, pp. 217-275. 

\bibitem{MW}
C.~M\oe glin, J.-L.~Waldspurger,
\emph{Sur l'involution de Zelevinski},
J. Reine Angew. Math. 372 (1986), 136--177. 


\bibitem{Re}
M.~Reeder,
\emph{Whittaker functions, prehomogeneous vector spaces and standard
  representations of $p$-adic groups},  J. Reine Angew. Math.  450
(1994), 83--121.  

\bibitem{Ze}
A.~Zelevinsk\u\i, 
\emph{The $p$-adic analogue of the Kazhdan-Lusztig conjecture},
Funktsional. Anal. i Prilozhen.  15  (1981), no. 2, 9--21. 

\bibitem{Ze1}
A.~Zelevinsk\u\i,
\emph{Two remarks on graded nilpotent classes},
Uspekhi Mat. Nauk 40 (1985), no. 1(241), 199--200. 

\bibitem{Ze3}
A.~Zelevinsky,
\emph{Induced representations of reductive ${p}$-adic groups. II. On irreducible representations of $GL(n)$}, Ann. Sci. \'Ecole Norm. Sup. (4)  13  (1980), no. 2, 165--210. 


\end{thebibliography}
\end{document}